\documentclass{amsart}       % onecolumn (second format)

\usepackage{graphicx}
\usepackage{amsmath}
\usepackage{nicefrac}
\usepackage{amsthm}
\usepackage{amssymb}
\usepackage{latexsym}
\usepackage{amscd}
\usepackage{mathtools}
\usepackage[cmtip,all]{xy}
\usepackage{eucal}
\usepackage{enumerate}
\newcommand{\C}{\mathcal{C}}
\newcommand{\F}{\mathcal{F}}
\newcommand{\Hcal}{\mathcal{H}}
\newcommand{\I}{\mathcal{I}}
\newcommand{\m}{{\mathfrak m}}
\newcommand{\p}{{\mathfrak p}}
\newcommand{\Q}{\mathbb{Q}}
\newcommand{\R}{\mathbb{R}}
\newcommand{\W}{\mathcal{W}}
\newcommand{\Z}{{{\mathbb{Z}}}}
\newcommand{\Zh}{{{\mathbb{Z}[\nicefrac{1}{2}]}}}
\newcommand{\ie}{{i.e.\,}}

\newcommand{\cf}{{c.f.\,}}
\newcommand{\sign}{{\rm sign\,}}
\newcommand{\sgn}{{\rm sgn\,}}
\newcommand{\isign}{{\rm sign^i\,}}
\newcommand{\shfsign}{{\mathcal{S}\textup{ign}}\,}
\newcommand{\sper}{{\rm Sper\,}}
\newcommand{\spec}{{\rm Spec\,}}
\newcommand{\supp}{{\rm Supp\,}}
\newcommand{\vcd}{{\rm vcd_2\,}}
\newcommand{\cd}{{\rm cd_2\,}}
\newcommand{\str}{{\rm st_r\,}}
\DeclareSymbolFont{bbold}{U}{bbold}{m}{n}
\DeclareSymbolFontAlphabet{\mathbbold}{bbold}
\newcommand{\ind}{{\mathbbold{1}}}
\DeclarePairedDelimiter\floor{\lfloor}{\rfloor}

\swapnumbers
\newtheorem{theorem}[equation]{Theorem}
\newtheorem{lemma}[equation]{Lemma}
\newtheorem{definition}[equation]{Definition}
\newtheorem{corollary}[equation]{Corollary}
\newtheorem{proposition}[equation]{Proposition}

\numberwithin{equation}{section}

\begin{document}

\title{From the global signature to higher signatures}
\thanks{I wish to thank Raman Parimala for bringing to my attention Max Karoubi's question and for suggesting to use the global signature to try to prove it.
}

%\titlerunning{On the global signature}        % if too long for running head

\author{Jeremy A. Jacobson}


\address{Department of Mathematics and Computer Science, Emory University, 400 Dowman Drive NE W401, Atlanta, GA, 30322, USA.}
        
              \email{jeremy.a.jacobson@emory.edu}           %  \\
           %\and
           %S. Author \at
            %  second address

\begin{abstract}
Let $X$ be an algebraic variety over the field of real numbers $\R$. We use the signature of a quadratic form to produce ``higher'' global signatures relating the derived Witt groups of $X$ to the singular cohomology of the real points $X(\R)$ with integer coefficients. We also study the global signature ring homomorphism and use the powers of the fundamental ideal in the Witt ring to prove an integral version of a theorem of Raman Parimala and Jean Colliot-Thelene on the mod 2 signature. Furthermore, we obtain an Atiyah-Hirzebruch spectral sequence for the derived Witt groups of $X$ with 2 inverted. Using this spectral sequence, we provide a bound on the ranks of the derived Witt groups of $X$ in terms of the Betti numbers of $X(\R)$. We apply our results to answer a question of Max Karoubi on boundedness of torsion in the Witt group of $X$. Throughout the article, the results are proved for a wide class of schemes over an arbitrary base field of characteristic different from 2 using real cohomology in place of singular cohomology. 
\end{abstract}

\maketitle

\section{Introduction}
\label{intro}
Let $X$ be a smooth, geometrically connected, quasi-projective variety over the field $\R$ of real numbers. Let $s$ denote the cardinality of the set of connected components of the real points $X(\R)$ for the Euclidean topology, and $d$ the Krull dimension of $X$. 

For any integer $n\geq0$, let $\Hcal^n$ denote the sheafification for the Zariski topology on $X$ of the presheaf $U\mapsto H^n_{\acute{e}t}(U,\Z/2\Z)$, where $H^n_{\acute{e}t}(U,\Z/2\Z)$ is the \'etale cohomology of $U$ with coefficients in the constant \'etale sheaf $\Z/2\Z$. 
In \cite{PCT}, Raman Parimala and Jean Colliot-Thelene introduced a mod 2 signature map $h_n$ and demonstrated that it induces an isomorphism of groups 
\begin{equation*}\label{eqn:PCT}
H^0_{Zar}(X,\Hcal^n)\stackrel{h_n}{\simeq} (\Z/2\Z)^s 
\end{equation*}
when $n\geq d+1$, where $H^0_{Zar}(X,\Hcal^n)$ denotes the sections of $\Hcal^n$. This generalized, from the case when $X$ is a curve and $n=2$, a classic theorem of Ernst Witt on the unramified Brauer group \cite{PCT}.

Let $H^p_{Zar}(X,\Hcal^n)$ denote the sheaf cohomology of the Zariski sheaf $\Hcal^n$ and let $H^p_{sing}(X(\R),\Z/2\Z)$ denote the singular cohomology with $\Z/2\Z$-coefficients. Their result was extended (\cite[(19.5.1)]{S94}, \cf J. Colliot-Thelene, Vortrag, Oberwolfach, 15 June 1990) to an isomorphism of cohomology groups for all $p\geq 0$
\begin{equation*}\label{eqn:S94}
H^p_{Zar}(X,\Hcal^n)\stackrel{h_n}{\simeq} H^p_{sing}(X(\R),\Z/2\Z)
\end{equation*}
when $n\geq d+1$. Claus Scheiderer later demonstrated that these isomorphisms hold for a very general class of schemes by using real cohomology in place of singular cohomology \cite[Corollary 19.5.1]{S94}.

One cannot obtain an integral version of the isomorphisms above by replacing everywhere $\Z/2\Z$ with $\Z$. When $n>d$ the \'etale cohomology groups $H^n_{\acute{e}t}(U,\Z)$ are torsion for any open subscheme $U$ of $X$ \cite[Corollary 7.23.3]{S94}. The sheafification $\Hcal^n(\Z)$ of $U\mapsto H^n_{\acute{e}t}(U,\Z)$ will produce a sheaf with $H^0_{Zar}(X,\Hcal^n(\Z))$ torsion, yet $\Z^s$ is torsion free.

We obtain integral versions of the above isomorphisms by using the Witt ring $W(X)$ of symmetric bilinear forms over $X$ (see \cite{KnebuschQueens} for an introduction) in place of \'etale cohomology. Let $\W$ denote the Zariski sheaf on $X$ associated to the presheaf $U\mapsto W(U)$. The sheaf $\W$ has a filtration by subsheaves $\I^n$ (see Section \ref{subsect:WittSheafPowers} for the definition of $\I^n$). The following theorem is a part of Theorem \ref{thm:B} in Section \ref{subsect:GlobalSignFundamentalIdeal}.
\begin{theorem}
If $X$ is an integral separated scheme of Krull dimension $d$ that is smooth and of finite type over $\R$, then there is an isomorphism of cohomology groups for all $p\geq 0$
 $$H^p_{Zar}(X,\I^n)\stackrel{\shfsign}{\simeq} H^p_{sing}(X(\R),2^n\Z)\stackrel{2^{-n}}{\simeq}H^p_{sing}(X(\R), \Z)$$ when $n\geq d+1$.
\end{theorem}  
When $X$ is an affine smooth $\R$-variety with trivial canonical sheaf, the case $p=d$ of this theorem was proved by Jean Fasel  \cite[Proposition 5.1]{FaselOrbitSets}. 

We obtain this theorem, as well as the other main results of the article, by studying a morphism $\shfsign$ of sheaves on $X$ that is induced by the global signature.  
Recall that in the situation of a real variety $X$, the global signature defines a ring homomorphism
\begin{equation*}
\sign: W(X)\rightarrow H^0_{sing}(X(\R),\Z)
\end{equation*}
mapping from the Witt ring $W(X)$ of symmetric bilinear forms to the singular cohomology with integer coefficients in degree zero. The most important result on the global signature was proved by Louis Mah\'e, who showed that the cokernel is 2-primary torsion (in fact, he proved this for any affine scheme) \cite[Corollary 4.10]{Mahe}.

Here we extend the global signature to ``higher signatures''
\begin{equation*}
\isign:W^i(X)\rightarrow H^i_{sing}(X(\R),\Z)
\end{equation*}
for $0\leq i\leq 3$, mapping from the derived Witt groups of $X$ to the singular cohomology of $X(\R)$. See Section \ref{subsec:DefinitionHigherSignature} for the definition. Recall that the derived Witt groups are $4$-periodic, $W^i(X)\simeq W^{i+4}(X)$ and form a representable cohomology theory for varieties \cite{hornbostel}. In degree zero, $W^0(X)\simeq W(X)$ is the Witt ring of symmetric bilinear forms over $X$.

Up to 2-primary torsion, we express the kernel of the global and higher signatures in terms of singular cohomology by realizing them as edge maps in an Atiyah-Hirzebruch type spectral sequence for the Witt groups.
\begin{theorem}
If $X$ is a separated scheme of finite Krull dimension that is smooth and of finite type over $\R$, then there is a strongly convergent spectral sequence
\begin{equation*}
\textup{E}_2^{p,q}=\begin{cases} H^p_{sing}(X(\R),\Zh)& \textup{if } q\equiv 0 \textup{ mod } 4\\
						0&	\textup{otherwise}
			\end{cases} 
			\Longrightarrow W^{p+q}(X)[\nicefrac{1}{2}]
\end{equation*}
abutting to the derived Witt groups with 2 inverted. For $r\geq1$, the differentials $d_r$ are of bidegree $(r,r-1)$. For integers $0\leq i \leq 3$, the global and higher signatures $\isign$, after inverting 2, are edge maps in this spectral sequence.
\end{theorem}

When $d\leq 3$ the spectral sequence collapses on the $\textup{E}_2$-page, in which case we have the following corollary showing that the kernel and cokernel of the higher signatures are 2-primary torsion.
\begin{corollary}
If $d\leq 3$, then the higher signature maps induce isomorphisms of groups after inverting the integer $2$ 
$$W^0(X)[\nicefrac{1}{2}] \stackrel{\textup{sign}^0}{\simeq} H^0_{sing}(X(\R),\Zh)$$
$$W^1(X)[\nicefrac{1}{2}] \stackrel{\textup{sign}^1}{\simeq} H^1_{sing}(X(\R),\Zh)$$
$$W^2(X)[\nicefrac{1}{2}] \stackrel{\textup{sign}^2}{\simeq} H^2_{sing}(X(\R),\Zh)$$
$$W^3(X)[\nicefrac{1}{2}] \stackrel{\textup{sign}^3}{\simeq} H^3_{sing}(X(\R),\Zh)$$
\end{corollary}
As another corollary to this spectral sequence we obtain a bound on the ranks of the Witt groups in terms of the Betti numbers of $X(\R)$ (See Corollary \ref{cor:Betti}).

Our results have applications to studying torsion in the Witt ring. Recall that in \cite[Example and Remark]{Karoubi}, Max Karoubi gave examples of rings $A$ for which the Witt group of the ring has odd torsion. The rings are continuous real and complex-valued functions on lens spaces of dimension 5 and higher, where we mean dimension as a real manifold. We prove the following corollary to the spectral sequence and show that the bound on the dimension is sharp (See Corollary \ref{cor:4Manifold2Primary}). 
\begin{corollary}
If $d\leq 4$ and $X(\R)$ is compact, then torsion in $W(X)$ is 2-primary.
\end{corollary}

Also, M. Karoubi has asked if the order of the torsion elements in the Witt group of a real variety is bounded. He and Charles Weibel have a work in progress in which they answer this question affirmatively with precise bounds. We provide an answer to this question by giving sufficient conditions for torsion to be bounded over a general base field $F$. For the proof of the following theorem see Section \ref{sec:BoundingTorsion}.  
\begin{theorem}
Let $F$ be a field of characteristic different from 2. Let $X$ be an excellent, integral, noetherian, regular, separated $F$-scheme of finite Krull dimension and of finite virtual cohomological 2-dimension. If the real cohomology groups $H^p(X_r, \Z)$ are finitely generated, then the torsion in the Witt groups $W^i(X)$ is of bounded order for all integers $i\in \Z$.
\end{theorem}
For an algebraic variety over a real closed field $R$ the groups $H^p(X_r, \Z)$ are finitely generated (See Section \ref{subsubsect:RealCoh}).
\begin{corollary}
If $X$ is an integral, separated scheme that is smooth and of finite type over a real closed field $R$, then torsion is bounded in the Witt group $W(X)$. 
\end{corollary}
Throughout the article, we use real cohomology in place of singular cohomology and work with a general class of schemes over an arbitrary base field of characteristic different from 2. All of the main results introduced above are proved in this generality throughout the article.

\section{The total signature}
\label{sec:signature}
The purpose of this section is to recall the definition of the signature following the text of Knebusch and Scheiderer on the subject of real algebra \cite{KnebuschScheiderer} and then recall facts about the signature which will be ``globalized'' to sheaves in Section \ref{sec:SignSheaves} and used throughout the paper.

\subsection{Orderings}
\label{subsect:ordering}
Let $F$ be a field. An ordering on $F$ is a subset $P\subset F$ satisfying the following:
\begin{enumerate}[(i)]
\item $P+P\subset P$, $PP\subset P$;
\item $P\cap (-P)=0$;
\item $P\cup -P=F$.
\end{enumerate}
The pair $(F,P)$ is called an ordered field. Given any such ordering $P$, one defines $a\leq b$ if $b-a\in P$. It follows from the axioms that if $F$ is nontrivial, then $1>0$.

Given an ordering $P$ on $F$, we define a function $\sgn_P:F^{*}\rightarrow\{1,-1\}$ by defining $\sgn_P(a)$ to be $1$ if $a\in P$ and $-1$ if $a\in -P$. Note that $\sgn_P(ab)=\sgn_P(a)\sgn_P(b)$ for $a,b \in F^{*}$.

\subsection{The real spectrum}
\label{subsect:Realspectrum}
We briefly recall some of the notation and definitions on the real spectrum. For the facts below see \cite[(0.4)]{S94}, \cite[\S 3, Definition 1]{KnebuschScheiderer}. The real spectrum of a ring $A$ is a topological space denoted by $\sper A$. As a set it consists of all pairs $\xi=(\p, P)$ with $\p \in \spec A$ and $P$ an ordering of the residue field $k(\p)$. For any point $\xi\in \sper A$, let $k(\xi)$ denote the real closure of the ordered field $k(\p)$ with respect to $P$. For $a\in A$, write $a(\xi)>0$ to indicate that the image of $a$ in $k(\xi)$ is positive. The sets of the form $D(a):=\{\xi \in \sper A: a(\xi)>0\}$, $a\in A$, form a subbasis of open sets for the topology on $\sper A$. The map $$\supp: \sper A\rightarrow \spec A$$ defined by $(\p, P) \mapsto \p$ is a continuous map of topological spaces called the \emph{support map}. 

The real spectrum of a scheme $X$ is the topological space $X_r$ formed by glueing the real spectra of its open affine subschemes. This does not depend on the open cover of $X$ that was chosen. Just as was defined for affine schemes, for any scheme $X$, there is a support map $\supp: X_r\rightarrow X$ which is a continuous map of topological spaces. 

\subsection{Definition of the total signature}
\label{subsect:sign}   
Let $F$ be a field. If $F$ has an ordering $P$, then any non-degenerate quadratic form $\phi$ over $F$ splits as an orthogonal sum $\phi\simeq \phi_{+}\perp \phi_{-}$, where the form $\phi_{+}$ is positive definite with respect to the ordering (for all $0\neq v$, $q(v)>0$ with respect to $P$) and the form $\phi_{-}$ is negative definite with respect to the ordering (\ie $-\phi_{-}$ is positive definite). The numbers $n_{+}:=\dim \phi_{+}$ and $n_{-}:=\dim \phi_{-}$ do not change under an isometry of $\phi$ \cite[Chapter 1, Section 2, Satz 2]{KnebuschScheiderer}. The integer $\sign_{P}([\phi]):=n_{+}-n_{-}$ is defined to be \emph{signature of $[\phi]$ with respect to $P$}. As the signature of the hyperbolic form is trivial, assigning an isometry class $[\phi]$ to its signature $\sign_{P}([\phi])$ defines a map 
\begin{equation*}
 \sign_{P}:W(F)\rightarrow \Z
\end{equation*}
which is a homomorphism of rings \cite[Chapter 1, Section 2, Satz 2]{KnebuschScheiderer}. 

Let $H^0(\sper F,\Z)$ (often written elsewhere $C(X_{F},\Z)$) denote the set of continuous integer valued functions. The \emph{total signature}
\begin{equation*}\label{eqn:Signature}
\sign:W(F)\rightarrow H^0(\sper F,\Z)
\end{equation*}
assigns an isometry class $[\phi]$ to the continuous function $P\mapsto \sign_{P}([\phi])$. It is a ring homomorphism since $\sign_{P}$ is. If $F$ has no ordering, then $\sign$ is trivial. 

\begin{lemma}
\label{lem:SignPfister}
Let $F$ be a field and let $a\in F^{*}$.
\begin{enumerate}[(i)]
\item For any ordering $P$ on $F$, the signature of the rank one form $\langle a\rangle$ with respect to $P$ is
$$\sign_P(\langle a \rangle)=\sgn_P(a)$$
\item The total signature of the Pfister form $\langle\langle a\rangle\rangle$ is 
$$\sign(\langle\langle a\rangle\rangle)=2\ind_{\{a<0\}}$$
\item The total signature of the $n$-fold Pfister form $\langle\langle a_1, \cdots, a_n\rangle\rangle$ is
$$\sign(\langle\langle a_1, \cdots, a_n\rangle\rangle)=2^n\ind_{\{a_1<0, \cdots, a_n<0\}}$$
 \end{enumerate}
\end{lemma}
\begin{proof}
The first statement follows from the definition of the signature. To prove the equality in the second statement we
use the definition and the fact that the signature is a ring homomorphism:
\begin{eqnarray*}
 \sign(\langle\langle a\rangle\rangle)
 &=&\sign(\langle 1,-a\rangle)\\
 &=&\sign(\langle 1 \rangle)-\sign(\langle a \rangle)\\
 &=&(\ind_{\{1>0\}}-\ind_{\{1<0\}})-(\ind_{\{a>0\}}-\ind_{\{a<0\}})\\
 &=&1-\ind_{\{a>0\}}+\ind_{\{a<0\}}\\
 &=& \begin{cases}
 0 & \textup{if } a>0\\
 2 & \textup{if } a<0
 \end{cases}\\
 &=& 2\ind_{\{a<0\}}
 \end{eqnarray*}
To prove the third statement, use $(ii)$ and the fact that $\sign$ is a ring homomorphism. 
\end{proof}

\subsection{The signature and the fundamental ideal}
\label{subsect:SignatureAndI^n}
As hyperbolic forms have even rank, assigning a quadratic form to its rank modulo 2 determines a ring homomorphism $e:W(F)\rightarrow \Z/2\Z$. The kernel of $e$ is denoted $I(F)$ and is called the fundamental ideal of $F$. The powers of the fundamental ideal $I^j(F)$ are additively generated by Pfister forms $\langle\langle a_1, \cdots, a_j \rangle\rangle$, so it follows from Lemma \ref{lem:SignPfister} that the signature induces a group homomorphism
\begin{equation*}
\sign:I^j(F)\rightarrow H^0(\sper F, 2^j\Z)
\end{equation*}

The diagram below commutes
\begin{equation*}
\xymatrix{
I^j(F)\ar^{\sign \hspace{7mm}}[r] \ar^2[d] & H^0(\sper F, 2^j\Z) \ar^2[d] \\
I^{j+1}(F)\ar^{\sign\hspace{7mm}}[r] & H^0(\sper F, 2^{j+1}\Z) 
}
\end{equation*} 
so after identifying
\begin{equation*}
\varinjlim(H^0(\sper F ,\Z)\stackrel{2}{\rightarrow}H^0(\sper F ,2\Z)\stackrel{2}{\rightarrow}H^0(\sper F ,2^2\Z)\stackrel{2}{\rightarrow}\cdots )\simeq H^0(\sper F ,\Z)
\end{equation*}
one obtains a homomorphism of groups. 
\begin{equation}\label{eqn:I^nColimitIso}
\varinjlim(W(F)\stackrel{2}{\rightarrow}I(F)\stackrel{2}{\rightarrow}I^2(F)\stackrel{2}{\rightarrow}\cdots )\stackrel{\sign}{\rightarrow} H^0(\sper F, \Z)
\end{equation}
%\begin{equation}\label{eqn:I^nSignSequence}
%0\rightarrow colim I^n_{tors}(F)\rightarrow colim I^n(F)\stackrel{\sign_n}{\rightarrow}colim H^0(\sper F, 2^n\Z) \rightarrow colim cok \sign_n \rightarrow 0.
%\end{equation}
 
The following is a theorem due to J. Arason and M. Knebusch \cite[Satz 2a.]{ArasonKnebusch}.
\begin{proposition}\label{prop:ColimitI^n}
The morphism \eqref{eqn:I^nColimitIso} is an isomorphism.
\end{proposition}

Since the cokernel of $\sign$ is $2$-primary torsion \cite[p.34, Theorem 3.4]{Lam}, one defines the reduced stability index $\str(F)$ to be the integer $n$ if the cokernel of $\sign$ has exponent $2^n$ and writes $\str(F)=\infty$ otherwise.

Recall that for any field $F$, the \emph{virtual cohomological $2$-dimension} is defined to be
\begin{equation*}
\vcd (F)=\cd(F(\sqrt{-1}))
\end{equation*}
where $\cd(F(\sqrt{-1}))$ denotes the cohomological $2$-dimension of the field $F(\sqrt{-1})$.

The following Lemma collects in a way we find convenient known results. 
\begin{lemma}\label{lem:EquivalentI^nProperties}
Let $F$ be a field of characteristic different from 2 and $s\geq 0$ an integer. The following are equivalent:
\begin{enumerate}[(i)]
\item $\vcd(F)\leq s$;
\item for $j\geq s+1$, $I^{j}(F(\sqrt{-1}))=0$;
\item for $j\geq s+1$, $I^{j}(F)$ is torsion free and $2I^{j-1}(F)=I^{j}(F)$;
\item for $j\geq s+1$, $I^{j}(F)$ is torsion free and $\str(F)\leq s$.
\item for $j\geq s+1$, $\sign_j:I^{j}(F)\rightarrow H^0(\sper F, 2^{j}\Z)$ is an isomorphism of rings.
\end{enumerate}
\end{lemma}
\begin{proof}
For lack of reference, we prove the equivalence of $(i)$ and $(ii)$. 
To prove that $(i)$ implies $(ii)$, use the Milnor conjecture, as proved in \cite[Theorem 4.1]{Voevodsky}, to obtain that for every integer $n\geq0$ there is a short exact sequence
\begin{equation*}
0\rightarrow I^{j+1}(F(\sqrt{-1}))\rightarrow I^{j}(F(\sqrt{-1}))\rightarrow H^{j}_{Gal}(F(\sqrt{-1}),\Z/2\Z)
\end{equation*} 
So $(i)$ implies that $H^{j}_{Gal}(F(\sqrt{-1}),\Z/2\Z)$ is trivial for $j\geq s+1$. Then $I^{j}(F(\sqrt{-1}))\simeq \underset{k\geq j}{\cap} I^k(F(\sqrt{-1}))$, and the intersection is trivial by the Arason-Pfister Haupsatz. 
We have that $(ii)$ implies $(i)$ since $I^{j}(F(\sqrt{-1}))$ surjects onto $H^{j}_{Gal}(F(\sqrt{-1}),\Z/2\Z)$.

For the proof of the equivalence of $(ii)$ and $(iii)$, see \cf \cite[Corollary 35.27 (1) and (4)]{MerkKarpElma}.

For the remaining equivalences, we may assume that $F$ is formally real. Indeed, if $F$ is not formally real, then $W(F)$ is 2-primary torsion and $\sper F=\emptyset$. In which case, the remaining equivalences of the lemma become clear. The equivalence of $(iii)$ and $(iv)$ was proved by L. Br\"ocker \cite[ZurTheoreiDerquad, Satz 3.17]{Brocker}. The equivalence of $(iv)$ and $(v)$ by J. Arason and M. Knebusch \cite[p.184]{ArasonKnebusch}.  
\end{proof}

\section{Sheaves and residues}

\subsection{The Witt sheaf and its filtration by powers of the fundamental ideal}\label{subsect:WittSheafPowers}
Let $X$ be a scheme. Let $\W$ denote the sheafification for the Zariski topology on $X$ of the presheaf $U\mapsto W(U)$, where $W(U)$ is the Witt ring of symmetric bilinear forms over $U$ as defined in Knebusch \cite{KnebuschQueens}. 

Now let $F$ be a field of characteristic different from 2, and assume that $X$ is an integral, regular, noetherian, separated $F$-scheme.
Under these hypotheses on $X$, the Gersten conjecture is known \cite[Theorem 6.1]{GillePaninBalmerWalter}. It follows that the sections of the Witt sheaf $\W(U)$ over any open subscheme $U$ in $X$ is a subgroup of $W(K)$, where $K$ is the function field of $X$ and $W(K)$ is the classical Witt ring of the field $K$. For $j\geq 0$, define $\I^j$ to be the Zariski presheaf on $X$ that assigns to every open $U$ the pullback
\begin{equation*}
\xymatrix{
\W(U) \ar[r] & W(K) \\
\I^j(U) \ar[r] \ar[u] & I^j(K) \ar[u]
}
\end{equation*}
In fact, $\I^j$ is a sheaf on $X$. This follows by using the definition of $\I^j$ and the fact that $\W$ is a sheaf.

\subsection{The second residue homomorphism}
\label{subsec:secondresidue}
The next lemma follows from a result of P. Balmer and C. Walter on the coniveau spectral sequence \cite[Lemma 8.4 (a)]{BalmerWalter}.
\begin{lemma}
\label{lem:WittSheafEqualsKernel}
Let $F$ be a field of characteristic different from 2 and let $X$ be an integral, noetherian, regular, separated, $F$-scheme. For each $x\in X^{(1)}$ choose a uniformizing parameter $\pi$. For any open subscheme $U$ in $X$, the sections of the Witt sheaf over $U$ are $$\W(U):=\textup{ker}(W(K)\stackrel{\oplus\partial_{\pi}}{\rightarrow}\bigoplus_{x\in U^1}W(k(x)))$$ where $\partial_{\pi}$ is, for every $x\in U^1$, the second residue morphism for $\mathcal{O}_{X,x}$ and $K$ is the function field of $X$. In particular, the kernel of $W(K)\stackrel{\oplus\partial_{\pi}}{\rightarrow}\bigoplus_{x\in U^1}W(k(x))$ does not depend on the choice of uniformizing parameters.
\end{lemma}
By restricting the second residue from $W(K)$ to $I^j(K)$ one obtains a map $$I^j(K)\rightarrow \underset{x\in U^{(1)}}{\bigoplus} I^{j-1}\left(k\left(x\right)\right)$$ because the second residue maps respect the powers of the fundamental ideal \cite[Theorem 6.6]{Gille07}.
\begin{lemma}
\label{lem:KernelIjresidue}
Let $X$ satisfy the hypotheses of Lemma \ref{lem:WittSheafEqualsKernel} and let $j\geq 0$ be an integer. For every open subscheme $U$ in $X$,
$$\I^j(U)=\textup{ker}(I^j\left(K\right) \stackrel{\oplus\partial_{\pi}}{\rightarrow} \underset{x\in U^{(1)}}{\bigoplus} I^{j-1}\left(k\left(x\right)\right))
$$
\end{lemma}
\begin{proof}
 By Lemma \ref{lem:WittSheafEqualsKernel} the kernel 
$$\textup{ker}(W\left(K\right) \stackrel{\oplus\partial_{\pi}}{\rightarrow} \underset{x\in U^{(1)}}{\bigoplus} W\left(k\left(x\right)\right))
$$
equals $\W(U)$. Then we have that $\I^j(U)$ equals the restriction of the kernel of $\oplus\partial_{\pi}$ to $I^j(K)$ since, by definition, $\I^j(U)$ is the pullback of $\W(U)$ over $I^j(K)$.
\end{proof}

The following lemma restates well-known facts on the second residue (\cf \cite[Chap. IV (1.2)-(1.3)]{MilnorHuse}.

\begin{lemma}
\label{lem:residueWittOnElements}
Let $A$ be a discrete valuation ring with fraction field $K$. Choose a uniformizing parameter $\pi$ for $A$. Then:
\begin{enumerate}[(i)]
\item every rank one quadratic form over $K$ is isometric to some $\langle c \rangle$, where $c=b\pi^{n}$, $b$ is a unit in $A$, and either $n=0$ or $n=1$;  
\item the second residue $\partial_{\pi}$ has the following description
\begin{equation*}
\partial_{\pi}(\langle c\rangle)=
\begin{cases}
\langle \overline{b}\rangle & \textup{ if $n=1$} \\
0 & \textup{ if $n=0$}
\end{cases}
\end{equation*}
on rank one forms $\langle c\rangle$ as in $(i)$.
\end{enumerate}
\end{lemma}

\subsection{Cohomology of the real spectrum}
\label{subsect:RealCoh}
In \cite{S94}, C. Scheiderer developed a theory of \emph{real cohomology} for schemes. It ``globalizes'' to all schemes the singular cohomology of the real points of a real variety in the same way that \'etale cohomology globalizes the singular cohomology of the complex points of a complex variety. We recall the definition and some properties we will need following \cite{S94}.\\

\subsubsection{Definition of real cohomology}\label{subsubsect:RealCoh}
Let $X$ be a scheme. First we recall the definition of the \emph{real site} of $X$, which we will also denote by $X_r$. It is the category $\textup{O}(X_r)$ of open subsets of $X_r$ equipped with the ``usual'' coverings, \ie a family of open subspaces $\{U_\lambda\rightarrow U\}$ is a covering of $U\in \textup{O}(X_r)$ if $U=\cup U_{\lambda}$. The category of sheaves of abelian groups on $X_r$ is denoted $\textup{Ab}(X_r)$. 

For any abelian group $A$, we will also denote by $A$ the sheaf on $X_r$ obtained by sheafifying the presheaf $U\mapsto A$, $U$ any open in $X_r$. Such a sheaf is called a \emph{constant sheaf}.

Let $Ab$ denote the category of abelian groups. For any sheaf $\F$ on $X_r$, the \emph{real cohomology groups of $X$ with coefficients in $\F$} are the right derived functors of the global sections functor $\Gamma:\textup{Ab}(X_r)\rightarrow \textup{Ab}$. They are denoted by 
$$H^p(X_r,\F):=R^p\Gamma\F$$
where $R^p\Gamma$ is the $p$-th derived functor of $\Gamma$. When $X=\spec A$ is affine, we will often write $H^p(\sper A, \F)$ instead of $H^p(X_r,\F)$.

For any separated scheme of finite type over $\R$ (or over any real closed field, with semi-algebraic cohomology replacing singular cohomology), the category of sheaves on $X(\R)$, where $X(\R)$ is equipped with the Euclidean topology, is equivalent to the category of sheaves on $X_r$ \cite[15.4.1]{S94}. Let $M$ be an abelian group. Since the sheaf cohomology of $X(\R)$ with coefficients $M$ coincides with the singular cohomology with coefficients $M$ \cite[15.1 and 15.3.2]{S94}, it follows that the real cohomology groups $H^p(X_r,M)$ coincide with the singular cohomology groups $H^p(X(\R),M)$. Hence, they are finitely generated groups. 

\subsubsection{Scheiderer's Gersten type resolution for real cohomology}
Next we recall the work of C. Scheiderer \cite{Scheiderer95} in which he constructs a ``Bloch-Ogus'' style complex that computes real cohomology. The codimension of support filtration on $X$ determines a coniveau spectral sequence for real cohomology. Scheiderer shows that for regular excellent schemes the $\text{E}_1$-page is zero except for the complex $C^{\bullet}(W,\F):=\text{E}_1^{*,0}$ and hence obtains the result below. Recall that a locally noetherian scheme is called excellent if $X$ can be covered by open affine subschemes $\spec A_{\alpha}$ where the $A_{\alpha}$ are excellent rings \cite[7.8.5]{EGAIV1}. For a point $x\in X$ of a scheme, we will denote $\sper k(x)$ by $x_r$.
\begin{proposition}\label{prop:complexReal}\cite[2.1 Theorem]{Scheiderer95}
Let $X$ be a regular excellent scheme. Let $W$ be an open constructible subset of $X_r$, and let $\F$ be a locally constant sheaf on $W$. Then there is a complex $C^{\bullet}(W,\F)$ of abelian groups
\begin{equation}
0\rightarrow \bigoplus_{x\in X^{(0)}}H^0_x(W,\F)\rightarrow \bigoplus_{x\in X^{(1)}}H^1_x(W,\F)\rightarrow \bigoplus_{x\in X^{(2)}}H^2_x(W,\F)\rightarrow \cdots ,
\end{equation}
natural in $W$ and $\F$, whose $q$th cohomology group is canonically isomorphic to $H^q(W, \F)$, $q
\geq 0$. Here $H^q_x(W,\F):=H^q_{x_r\cap W}(\sper \mathcal{O}_{X,x}\cap W, \F)$. This complex is contravariantly functorial for flat morphisms of schemes.
\end{proposition}

\subsection{The support sheaf}
Let $\textup{Ab}(X_{Zar})$ denote the category of sheaves of abelian groups on the Zariski site $X_{Zar}$. Since the support map $\supp:X_r\rightarrow X$ is a continuous map of topological spaces it induces the direct image functor $\supp_{*}:\textup{Ab}(X_r)\rightarrow \textup{Ab}(X_{Zar})$. This functor is faithful and exact \cite[Theorem 19.2]{S94}. 
\begin{lemma}\label{lem:ZariskiCohomologySheafSuppZ}
Let $X$ be a scheme. For any sheaf $\F\in \textup{Ab}(X_r)$, 
 $$H^p(X_r,\F)\simeq H^p_{Zar}(X,\supp_{*}\F)$$
\end{lemma}
\begin{proof}
Using the Grothendieck spectral sequence for the composition of the functors $\supp_{*}$ and the global sections functor $\Gamma$ we obtain a spectral sequence with $\text{E}_2^{p,q}=H^p_{Zar}(X, R^q\supp_{*}\F)$ that abuts to $H^{p+q}(X_r,\F)$. For $q>0$, the sheaves $R^q\supp_{*}\F$ vanish \cite[Theorem 19.2]{S94}. Therefore the edge maps in this spectral sequence determine isomorphisms $H^p(X_r,\F)\stackrel{\simeq}{\rightarrow}H^p_{Zar}(X,\supp_{*}\F)$ for $p\geq 0$.
\end{proof}

\subsection{A residue-type homomorphism for real cohomology}
\begin{definition}\label{def:Beta}
Let $(B,\m)$ be a discrete valuation ring, let $K$ denote the fraction field of $B$ and let $\pi$ be a uniformizing parameter for $B$. Let $A$ be an abelian group equipped with the discrete topology. For any point $\xi \in \sper B/\m$, there are exactly two orderings on $K$ which induce the ordering of $\xi$ on $B/\m$ and they differ by their sign on the uniformizing parameter \cite[Theorem 10.1.10, see Definition 10.1.7 for induced orderings]{Coste}. Denote them by $\eta_{+}$ and $\eta_{-}$, where $\sgn_{\eta_{+}}(\pi)=1$ and $\sgn_{\eta_{-}}(\pi)=-1$. The group homomorphism
\begin{equation*}
\beta_{\pi}: H^0(\sper K, A) \rightarrow H^0(\sper B/\m, A)
\end{equation*}
is defined by assigning $s\in H^0(\sper K, A)$ to the map $\xi\mapsto \beta_{\pi}(s)(\xi):=s(\eta_{+})-s(\eta_{-})$. If $\sper B/\m=\emptyset$, then it is defined to be zero.
\end{definition}

\begin{lemma}\label{lem:BetaMorphism}
Let $B$ be a discrete valuation ring and $\pi$ be a uniformizing parameter for $B$.
The morphism $\beta_{\pi}$ of Definition \ref{def:Beta} has the following description on elements $\sign(\langle c \rangle)$, where $c=b\pi^{n}$, $b$ is a unit in $B$, and either $n=0$ or $n=1$:
\begin{equation*}
\beta_{\pi}(\sign(\langle c \rangle))=
\begin{cases}
2\sign(\langle\overline{b}\rangle) & \textup{ if $n$ is 1} \\
0 & \textup{ if $n$ is 0}
\end{cases}
\end{equation*}
\end{lemma}
\begin{proof}
Let $c=b\pi^{n}$, where $b$ is a unit in $B$, and either $n=0$ or $n=1$. We have the following equalities which will prove the lemma. For any $\xi \in \sper B/\mathbf{m}$:  
\begin{eqnarray*}
\beta_{\pi}(\sign(\langle c\rangle))(\xi)
&=&\sign_{\eta_+}(\langle c \rangle)-\sign_{\eta_{-}}(\langle c \rangle)\\
&=&\sgn_{\eta_+}(c)-\sgn_{\eta_{-}}(c)\\
&=&\begin{cases}
\sgn_{\xi}(\overline{c}) - \sgn_{\xi}(\overline{c}) & \textup{if } n=0 \textup{ (both orderings induce } \xi) \\
\sgn_{\eta_{+}}(b\pi) - \sgn_{\eta_{-}}(b\pi)& \textup{if } n=1\\
\end{cases}\\
&=&\begin{cases}
0 & \textup{if } n=0 \\
\sgn_{\eta_{+}}(b)\sgn_{\eta_{+}}(\pi) - \sgn_{\eta_{-}}(b)\sgn_{\eta_{-}}(\pi)& \textup{if } n=1\\
\end{cases}\\
&=&\begin{cases}
0 & \textup{if } n=0 \\
\sgn_{\eta_{+}}(b)+\sgn_{\eta_{-}}(b) & \textup{if } n=1 \textup{ (by definition of }\eta_+ \textup{ and } \eta_{-})\\
\end{cases}\\
&=&\begin{cases}
0 & \textup{if } n=0 \\
\sgn_{\xi}(\overline{b}) + \sgn_{\xi}(\overline{b})& \textup{if } n=1 \textup{ (both orderings induce } \xi)\\
\end{cases}\\
&=& \begin{cases}
0 & \textup{if } n=0 \\
2\sgn_{\xi}(\overline{b}) & \textup{if } n=1\\
\end{cases}\\
\end{eqnarray*} 
\end{proof}

The following Lemma is based on the proof of \cite[2.6 Proposition]{Scheiderer95} where $A=\Z/2\Z$. 
\begin{lemma}\label{lem:BetaResidue}
Let $X$ be a regular excellent scheme which is integral with function field $K$. Let $x\in X^{(1)}$, let $\pi$ denote a choice of uniformizing parameter for $\mathcal{O}_{X,x}$, and let $A$ be a constant sheaf on $X_r$. Denote by $\partial$ the map
\begin{equation*}
H^0(\sper K, A)\rightarrow H^1_{x_r}(\sper \mathcal{O}_{X,x},A)
\end{equation*} induced by first differential of the complex $C^{\bullet}(X_r,A)$ (see Proposition \ref{prop:complexReal}). Then, there is an isomorphism $\iota_{\pi}:H^1_{x_r}(\sper \mathcal{O}_{X,x}, A)\rightarrow H^0(x_r, A)$ for which $\iota_{\pi}(\partial)=\beta_{\pi}$.
\end{lemma}
\begin{proof}
Let $X^{'}=\sper \mathcal{O}_{X,x}$, $Z^{'}=x_r$, let $i:Z^{'}\rightarrow X^{'}$ denote the inclusion, and let $j:\sper K \rightarrow X^{'}$ denote the inclusion of the complement to $Z^{'}$. 
For any abelian sheaf $A$ on $X^{'}$ the sequence
\begin{equation*}\label{eqn:R1sequence}
A\rightarrow j_{*}j^{*}A \rightarrow i_{*}R^1i^{!}(A) \rightarrow 0
\end{equation*} 
is exact (\cf \cite[8.2.4 Corollary, Chapter II]{Tamme}).
By \cite[Lemma 1.3]{Scheiderer95}, for any locally constant sheaf $A$ on $X^{'}$ the sequence
\begin{equation*}
A\rightarrow j_{*}j^{*}A \stackrel{\beta}{\rightarrow} i_{*}i^{*}A \rightarrow 0
\end{equation*}
is exact, where $\beta$ is defined on stalks $\zeta\in Z^{'}$ as $(\beta s)_{\zeta}=s(\eta_{+})-s(\eta_{-})\in A$. Hence we get an isomorphism $\iota_{\pi}$ of cokernels and a commutative diagram
\begin{equation}\label{eqn:DiagramOfLemma}
\xymatrix{
j_{*}j^{*}A(X^{'}) \ar^{\partial}[r] \ar^{\beta}[dr] & i_{*}R^1i^{!}A(X^{'}) \ar^{\iota_{\pi}}[d]\\
 & i_{*}i^{*}A(X^{'})
}
\end{equation}
Tracking down all the definitions, one finds that Diagram \eqref{eqn:DiagramOfLemma} is equal to the diagram below.
\begin{equation*}
\xymatrix{
H^0(X^{'}-Z^{'},A)\ar^{\partial}[r] \ar^{\beta_{\pi}}[dr] & H^1_{Z^{'}}(X^{'},A) \ar^{\iota_{\pi}}[d]\\
 & H^0(Z^{'},A)
}
\end{equation*}
where the vertical map is the isomorphism $\iota_{\pi}$ chosen, the diagonal map is the map $\beta_{\pi}$ of Definition \ref{def:Beta}, and $\sper K$ equals $X^{'}-Z^{'}$. This finishes the proof of the lemma.
\end{proof}

\begin{lemma}\label{lemma:KernelIsBeta}
Let $X$ be a regular excellent scheme which is integral with function field $K$ and let $A$ be a constant sheaf on $X_r$. For any $x\in X^{(1)}$, let $\pi_{x}$ be a choice of uniformizing parameter for $\mathcal{O}_{X,x}$. Then, for any open subscheme $U$ in $X$,
 \begin{equation*}
 \supp_{*}A(U)\simeq \textup{ker}( H^0(\sper K, A) \stackrel{\oplus\beta_{\pi}}{\rightarrow} \underset{x\in U^{(1)}}{\bigoplus} H^0(\sper k(x), A))
 \end{equation*}
where $\beta_{\pi}: H^0(\sper K, A) \rightarrow H^0(\sper k(x), A)$ is the map of Definition \ref{def:Beta}.
\end{lemma}
\begin{proof}
We may choose isomorphisms $\iota_{\pi}$ for each $x\in X^{(1)}$ as in Lemma \ref{lem:BetaResidue}. Using Proposition \ref{prop:complexReal} and Lemma \ref{lem:ZariskiCohomologySheafSuppZ} we obtain, for every open $U$, an isomorphism of kernels from $\supp_{*}A(U)$ to $\textup{ker}( H^0(\sper K, A) \stackrel{\oplus\beta_{\pi}}{\rightarrow} \underset{x\in U^{(1)}}{\bigoplus} H^0(\sper k(x), A))$.
\end{proof}

\subsection{Commutativity of residues}
\begin{proposition}\label{prop:FundamentalIdealResidueCommute}
Let $X$ be an excellent, integral, noetherian, regular, separated, $F$-scheme and $j\geq 0$ an integer. For every open subscheme $U$ in $X$:
\begin{enumerate}[(i)]
\item the diagram below commutes
\begin{equation}\label{eqn:CommutesResidue}
\xymatrix{  
I^j\left(K\right) \ar^{\oplus \partial_{\pi}}[r] \ar^{\sign}[d] & \underset{x\in U^{(1)}}{\bigoplus} I^{j-1}\left(k\left(x\right)\right) \ar^{2\sign}[d]\\
H^0(\sper K, 2^jZ) \ar^{\oplus\beta_{\pi}}[r] & \underset{x\in U^{(1)}}{\bigoplus} H^0(\sper k(x), 2^j\Z)
}
\end{equation} 
where the maps $\oplus\partial_{\pi}$ and $\oplus\beta_{\pi}$ are the maps of Lemmas \ref{lem:KernelIjresidue} and \ref{lemma:KernelIsBeta};
\item the signature $\sign:I^j(K)\rightarrow H^0(\sper K,2^j\Z)$ induces a map of kernels $$\I^j(U)\rightarrow \supp_{*}2^j\Z(U)$$.
\end{enumerate}
\end{proposition}
\begin{proof}
To prove $(i)$, recall (Lemma \ref{lem:residueWittOnElements} $(i)$) that $W(K)$ is generated by rank one forms  $\langle c \rangle$, where $c=b\pi^{n}$, $b$ is a unit in $\mathcal{O}_{X,x}$, and either $n=0$ or $n=1$. Hence, it suffices to check commutativity on such forms. Using Lemma \ref{lem:BetaMorphism} we find that for each $x\in U^{(1)}$, 
\begin{equation*}
\beta_{\pi}(\sign(\langle c \rangle))=
\begin{cases}
2\sign(\langle\overline{b}\rangle) & \textup{ if $n$ is 1} \\
0 & \textup{ if $n$ is 0}
\end{cases}
\end{equation*}
On the other hand, using Lemma \ref{lem:residueWittOnElements} $(ii)$ we find that 
\begin{equation*}
\partial_{\pi}(\langle c\rangle)=
\begin{cases}
\langle \overline{b}\rangle & \textup{ if $n=1$} \\
0 & \textup{ if $n=0$}
\end{cases}
\end{equation*}
which finishes the proof of the $(i)$.

To prove $(ii)$, consider Diagram \eqref{eqn:CommutesResidue}. We may identify the kernel of the upper horizontal map with $\I^j(U)$ using Lemma \ref{lem:KernelIjresidue}, and we identify the kernel of the lower horizontal map with $\supp_{*}2^n\Z(U)$ using Lemma \ref{lemma:KernelIsBeta}.
\end{proof}

\subsection{Sheaves and inverting 2}
\begin{lemma}\label{lem:InvertingAnIntegerIsColimit}
Let $R$ be a ring and $M$ an $R$-module. For any element $r\in R$, the localization $M[\nicefrac{1}{r}]$ of the $R$-module $M$ with respect to the multiplicative set $S=\{1,r,r^2,r^3,\ldots\}$ is equal as an $R$-module to the direct limit $\varinjlim(M\stackrel{r}{\rightarrow}M\stackrel{r}{\rightarrow}\cdots)$.
\end{lemma}
\begin{proof}
Using the direct sum construction of the direct limit $\varinjlim(M\stackrel{r}{\rightarrow}M\stackrel{r}{\rightarrow}\cdots)$, the relations one finds are the same as the relations defining the localization $M[\nicefrac{1}{r}]$.
\end{proof}

The following Lemma recalls well-known facts on Grothendieck topologies, \cf \cite[3.2.3 i) and Theorem 3.11.1]{Tamme} and \cite[Proposition 3.2]{S94}.
\begin{lemma}\label{lem:ChangeOfCoefficients}
Let $X$ be a scheme and let $X_t$ be either the real site $X_r$ or the Zariski site $X_{Zar}$. Then:
\begin{enumerate}[(i)]
\item direct limits of sheaves $\varinjlim(\F_i)$ exist in $\textup{Ab}(X_t)$ and equal the sheaf associated to the presheaf $U\mapsto \varinjlim \F_i(U)$;
\item if $X$ is noetherian, then for any $p\geq 0$ and any open $U$ in $X_t$, the canonical morphism $\F_i\rightarrow \varinjlim\F_i$ induces an isomorphism
$$ \varinjlim H^p(U_t,\F_i)\rightarrow H^p(U_t,\varinjlim\F_i)$$
of groups.
\end{enumerate}
\end{lemma}
Combining Lemmas \ref{lem:InvertingAnIntegerIsColimit} and \ref{lem:ChangeOfCoefficients}, we have the following lemma.
\begin{lemma}\label{lem:Inverting2CommutesWithCohomology}
Let $X$ be a scheme, let $X_t$ be either the real site $X_r$ or the Zariski site $X_{Zar}$, and let $\F\in \textup{Ab}(X_t)$. Let $\F[\nicefrac{1}{2}]$ denote the sheaf associated to the presheaf $U\mapsto \F(U)[\nicefrac{1}{2}]$. Then:
\begin{enumerate}[(i)]
\item the sheaf $\F[\nicefrac{1}{2}]$ equals $\varinjlim(\F\stackrel{2}{\rightarrow}\F\stackrel{2}{\rightarrow}\cdots)$;
\item if $X$ is noetherian, then for any $p\geq 0$ and any open $U$ in $X_t$, the canonical morphisms $\F\rightarrow \varinjlim(\F\stackrel{2}{\rightarrow}\F\stackrel{2}{\rightarrow}\cdots)$ induce an isomorphism
$$H^p(U_t,\F)[\nicefrac{1}{2}]\stackrel{\simeq}{\rightarrow} H^p(U_t,\F[\nicefrac{1}{2}])$$
of groups.
\end{enumerate}
\end{lemma}

\section{On the global signature}
\label{sec:SignSheaves}
\subsection{The global signature}
Let $X$ be a scheme. Recall that the underlying set of the real spectrum $X_r$ consists of pairs $(x,P)$, where $x \in X$ and $P$ is an ordering on the residue field $k(x)$. The \emph{global signature} is a ring homomorphism
\begin{equation*}
\sign:W(X)\rightarrow H^0(X_r,\Z)
\end{equation*}
from the Witt ring of symmetric bilinear forms over $X$ to the ring of continuous integer valued functions on $X_r$. It assigns an isometry class $[\phi]$ of a symmetric bilinear form $\phi$ over $X$ to the function on $X_r$ defined by $$\sign([\phi])(x,P):=\sign_P([i^{*}_x\phi])$$, where $i_x:x\rightarrow X$ is any point and $P$ is any ordering on $k(x)$. 

The most important result on the global signature was proved by L. Mah\'e. He proved that if $X=\spec A$ is affine, then the cokernel of the global signature is 2-primary torsion \cite[Th\'eor\`eme 3.2]{Mahe}. Equivalently, after inverting 2, \ie localizing with respect to the multiplicative set $S=\{1, 2,2^2,2^4,\cdots\}$, the global signature induces a surjection  
\begin{equation*}
\sign:W(A)[1/2] \twoheadrightarrow H^0(\sper A, \Z)[1/2] 
\end{equation*}
of rings. For $A$ a field this was well-known (\cf \cite[p.34, Theorem 3.4]{Lam})

When $A$ is a connected ring and $\sper A\neq \emptyset$, the kernel of the global signature is the nilradical in the Witt ring \cite[Section 1.3]{Mahe}. When $A$ is a connected local ring, either the nilradical is 2-primary torsion or the entire Witt ring $W(A)$ is 2-primary torsion \cite[Theorem 1.2]{KnebuschLocal}, and in either case the kernel is 2-primary torsion. When $A$ is a field, these statements on the kernel are Pfister's local-global principle \cite[Satz 22]{Pfister}.

\subsection{The global signature as a morphism of sheaves}
Let $X$ be a scheme. Recall that $\W$ denotes the Zariski sheaf on $X$ associated to the presheaf $U\mapsto W(U)$ and that we may identify $\supp_{*}\Z$ with the sheaf $U\mapsto H^0(U_r,\Z)$ (Lemma \ref{lem:ZariskiCohomologySheafSuppZ}). Hence the global signature induces a morphism 
\begin{equation*}
\shfsign:\W\rightarrow \supp_{*}\Z 
\end{equation*}
of Zariski sheaves on $X$ which we call the \emph{global signature morphism of sheaves}.

While the global signature is not in general an isomorphism after inverting 2, for instance, see Corollary \ref{cor:OddTorsionDerived} in Section \ref{subsect:AtiyahHirz}, the following theorem proves that the global signature morphism of sheaves is. Recall that if $X$ is a noetherian scheme, then for any Zariski sheaf $\F$ on $X$ the presheaf $U\mapsto \F(U)[\nicefrac{1}{2}]$ is a sheaf (Lemma \ref{lem:Inverting2CommutesWithCohomology} $(ii)$) that we denote by $\F[\nicefrac{1}{2}]$.
\begin{theorem}\label{thm:0}
Let $X$ be a regular, noetherian scheme. After inverting 2, the global signature morphism of sheaves $\mathcal{S}ign$ induces an isomorphism
\begin{equation*}
\shfsign:\W[\nicefrac{1}{2}]\stackrel{\simeq}{\rightarrow} \supp_{*}\Zh 
\end{equation*}
of Zariski sheaves on $X$ and an isomorphism
\begin{equation*}
H^p_{Zar}(X, \W[\nicefrac{1}{2}]) \stackrel{\simeq}{\rightarrow} H^p(X_r,\Zh)
\end{equation*}
of cohomology groups.
\end{theorem}
\begin{proof}
First we explain how to obtain the map $$\shfsign:\W[\nicefrac{1}{2}]\stackrel{\simeq}{\rightarrow} \supp_{*}\Zh$$ of sheaves. The global signature morphism of sheaves $\shfsign$ induces a morphism of presheaves from $U\mapsto \W(U)[\nicefrac{1}{2}]$ to $U\mapsto H^0(U_r,\Z)[\nicefrac{1}{2}]$. The former sheaf is $\W[\nicefrac{1}{2}]$ by definition. The latter sheaf coincides with $\supp_{*}\Z[\nicefrac{1}{2}]$ using Lemma \ref{lem:ZariskiCohomologySheafSuppZ}. %The group $H^0(U_r,\Z)[\nicefrac{1}{2}]$ is isomorphic to the group $H^p(U_r,\varinjlim(\Z\stackrel{2}{\rightarrow}\Z\stackrel{2}{\rightarrow}\cdots)$, where $\Z$ denotes the constant sheaf on $X_r$, by Lemma \ref{lem:Inverting2CommutesWithCohomology}. The sheaf $\varinjlim(\Z\stackrel{2}{\rightarrow}\Z\stackrel{2}{\rightarrow}\cdots)$ is equal to the constant sheaf $\Zh$ on $X_r$ since the stalk of the colimit $(\varinjlim(\Z\stackrel{2}{\rightarrow}\Z\stackrel{2}{\rightarrow}\cdots))_x$ is the colimit of the stalk $\varinjlim(\Z_x\stackrel{2}{\rightarrow}\Z_x\stackrel{2}{\rightarrow}\cdots)$. Finally, by Lemma \ref{lem:ZariskiCohomologySheafSuppZ} we have that the Zariski sheaf $U\mapsto H^0(U_r,\Zh)$ is isomorphic to $\supp_{*}\Zh$. 

To prove this map of sheaves is an isomorphism we prove that it is an isomorphism on stalks. For any point $x\in X$, this map induces on stalks
\begin{equation}\label{eqn:StalkMorphism}
\shfsign_x:W(\mathcal{O}_{X,x})[\nicefrac{1}{2}]\rightarrow H^0(\sper \mathcal{O}_{X,x}, \Z)[\nicefrac{1}{2}] 
\end{equation}
and this is the global signature, after inverting 2, of the scheme $\spec \mathcal{O}_{X,x}$. Since $X$ is regular, $\mathcal{O}_{X,x}$ is local and connected. Then, the facts on the kernel and cokernel of the global signature of connected local rings that were recalled in the previous section imply that \eqref{eqn:StalkMorphism} is an isomorphism.

Finally, the isomorphism in cohomology may be obtained from the isomorphism of sheaves after identifying $H^p_{Zar}(X,\supp_{*}\Z[\nicefrac{1}{2}])$ with $H^p(X_r,\Zh)$ using Lemmas \ref{lem:ZariskiCohomologySheafSuppZ} and \ref{lem:Inverting2CommutesWithCohomology}.
\end{proof}

\subsection{The global signature and the fundamental ideal}
\label{subsect:GlobalSignFundamentalIdeal}
\begin{definition}
Let $F$ be a field of characteristic different from 2. Let $X$ be an excellent, integral, noetherian, regular, separated, $F$-scheme. For $j\geq 0$ any integer, it follows from Proposition \ref{prop:FundamentalIdealResidueCommute} that $\sign:I^j(K)\rightarrow H^0(\sper K,2^j\Z)$ induces a morphism 
\begin{equation*}
\shfsign_j:\mathcal{I}^j\rightarrow \supp_{*} 2^j\Z
\end{equation*}
of Zariski sheaves on $X$.
\end{definition}

\begin{theorem}\label{thm:C}
Let $F$ be a field of characteristic different from 2. Let $X$ be an excellent, integral, noetherian, regular, separated, $F$-scheme. Then:
\begin{enumerate}[(i)] 
\item $\shfsign_j$ determines an isomorphism
\begin{equation*}
\varinjlim(\W\stackrel{2}{\rightarrow}\mathcal{I}\stackrel{2}{\rightarrow}\mathcal{I}^2\stackrel{2}{\rightarrow}\cdots )\stackrel{\simeq}{\rightarrow} \supp_{*}\Z
\end{equation*}
 of sheaves
\item define $\W_2:=\varinjlim(\W/\mathcal{I}\stackrel{2}{\rightarrow}\W/\mathcal{I}^2\stackrel{2}{\rightarrow}\W/\mathcal{I}^3\stackrel{2}{\rightarrow}\cdots)$. The short exact sequence of Zariski sheaves on $X$
\begin{equation*}
0\rightarrow \varinjlim(\W\stackrel{2}{\rightarrow}\I\stackrel{2}{\rightarrow}\I^2\stackrel{2}{\rightarrow}\cdots )\rightarrow \W[\nicefrac{1}{2}]\rightarrow \W_2\rightarrow 0
\end{equation*} 
is isomorphic via $\mathcal{S}ign$ to the short exact sequence of Zariski sheaves on $X$
\begin{equation*}
0\rightarrow\supp_{*} \Z \rightarrow \supp_{*} \Zh \rightarrow \supp_{*}\Z_2\rightarrow 0 
\end{equation*} 
where $\Z_2$ is the cokernel of the inclusion $\Z\rightarrow \Zh$.
\end{enumerate}
\end{theorem}
\begin{proof}
We prove $(i)$, and then $(ii)$ follows using Theorem \ref{thm:0}. For every open $U$ in $X$, from Proposition \ref{prop:FundamentalIdealResidueCommute} we have that the diagram below commutes 
\begin{equation*}
\xymatrix{  
I^j\left(K\right) \ar^{\oplus \partial_{\pi}}[r] \ar^{\sign}[d] & \underset{x\in U^{(1)}}{\bigoplus} I^{j-1}\left(k\left(x\right)\right) \ar^{2\sign}[d]\\
H^0(\sper K, 2^jZ) \ar^{\oplus\beta_\pi}[r] & \underset{x\in U^{(1)}}{\bigoplus} H^0(\sper k(x), 2^j\Z)
}
\end{equation*} 
Taking direct limits with respect to multiplication by 2, the left vertical arrow becomes an isomorphism by Proposition \ref{prop:ColimitI^n}, and the right vertical arrow is also an isomorphism because it is the composition $$I^{j-1}\left(k (x)\right)\stackrel{\sign}{\rightarrow} H^0(\sper k(x), 2^{j-1}\Z) \stackrel{2}{\rightarrow} H^0(\sper k(x), 2^j\Z)$$ and each of the maps in the composition are isomorphisms after taking direct limits. Direct limits are exact, therefore using Lemmas \ref{lem:KernelIjresidue} and \ref{lemma:KernelIsBeta} we get an isomorphism of kernels from
$$\varinjlim(\W(U)\stackrel{2}{\rightarrow}\mathcal{I}(U)\stackrel{2}{\rightarrow}\mathcal{I}^2(U)\stackrel{2}{\rightarrow}\cdots )$$ to $$\varinjlim(\supp_{*}\Z(U)\stackrel{2}{\rightarrow}\supp_{*}2\Z(U)\stackrel{2}{\rightarrow}\supp_{*}2^2\Z(U)\stackrel{2}{\rightarrow}\cdots )$$
for every open $U$. As $$\varinjlim(\supp_{*}\Z\stackrel{2}{\rightarrow}\supp_{*}2\Z\stackrel{2}{\rightarrow}\supp_{*}2^2\Z\stackrel{2}{\rightarrow}\cdots ) \simeq \supp_{*}\Z$$ we get the isomorphism of sheaves stated in $(i)$, and this finishes the proof.
\end{proof}

Define the virtual cohomological 2-dimension of $X$, denoted $\vcd(X)$, to be  $$\vcd(X)=sup\{\vcd(\eta)\}$$ where $\eta$ runs through the generic points of $X$.

Note that in $(ii)$ of the next theorem, the Milnor conjecture is used to obtain the short exact sequence of sheaves \eqref{eqn:MilnorSES}. This conjecture has been proved by V. Voevodsky \cite[Theorem 4.1]{Voevodsky}, there is also another proof by F. Morel \cf \cite{Morel}. 
\begin{theorem}\label{thm:B}
Let $F$ be a field of characteristic different from 2. Let $X$ be an excellent, integral, noetherian, regular, separated $F$-scheme. If $\vcd(X)$ is finite, say $\vcd(X)=s$, then:
\begin{enumerate}[(i)]
\item for $j\geq s+1$, $\mathcal{S}ign_j$ induces an isomorphism 
\begin{equation*}
\mathcal{I}^j\stackrel{\simeq}{\rightarrow} \supp_{*} 2^j\Z
\end{equation*}
of sheaves;
\item for $j\geq s+1$, the short exact sequence
\begin{equation}\label{eqn:MilnorSES}
\mathcal{I}^{j+1} \rightarrow \mathcal{I}^{j} \rightarrow \mathcal{H}^{j} 
\end{equation}
of sheaves is isomorphic via $\mathcal{S}ign_j$ to the short exact sequence
\begin{equation*}
\supp_{*} 2^{j+1}\Z \rightarrow \supp_{*} 2^j\Z \rightarrow \supp_{*}\Z/2\Z
\end{equation*}
of sheaves;
\item for $j\geq s+1$, $\mathcal{S}ign_j$ induces an isomorphism of cohomology groups $$H^p_{Zar}(X,\I^j)\simeq H^p(X_r,2^j\Z)$$ for all $p\geq 0$;
\item if $X$ is an integral, separated scheme that is smooth and of finite type over $\R$, then for $j\geq s+1$, the isomorphism of $(iii)$ determines an isomorphism of cohomology groups $$H^p_{Zar}(X,\I^j)\simeq H^p_{sing}(X(\R),2^j\Z)\stackrel{2^{-j}}{\simeq}H^p_{sing}(X(\R), \Z)$$ for all $p\geq 0.$
\end{enumerate} 
\end{theorem}
\begin{proof}
We prove $(i)$. The proof of $(ii)$ is an easy consequence of $(i)$, while $(iii)$ and $(iv)$ follow immediately from $(i)$. When $\vcd(X)=s$, for any $x\in X^{(p)}$ we have that $\cd(k(x)[\sqrt{-1}]) \leq \cd(X[\sqrt{-1}])-p$ \cite[Proposition 4.1 (a)]{Kahn}. Therefore, $\vcd(k(x))\leq s-p$. For every open subscheme $U$ of $X$, it follows that in the commutative diagram from Proposition \ref{prop:FundamentalIdealResidueCommute} 
\begin{equation*}
\xymatrix{  
I^j\left(K\right) \ar^{\oplus \partial_{\pi}}[r] \ar^{\sign}[d] & \underset{x\in U^{(1)}}{\bigoplus} I^{j-1}\left(k\left(x\right)\right) \ar^{2\sign}[d]\\
H^0(\sper K, 2^jZ) \ar^{\oplus\beta_{\pi}}[r] & \underset{x\in U^{(1)}}{\bigoplus} H^0(\sper k(x), 2^j\Z)
}
\end{equation*}
the vertical maps are isomorphisms for $j\geq s+1$ by Lemma \ref{lem:EquivalentI^nProperties}. Indeed, the right vertical map $I^{j-1}\left(k (x)\right)\stackrel{2\sign}{\rightarrow} H^0(\sper k(x), 2^{j}\Z)$ is the composition of the isomorphism $$I^{j-1}\left(k (x)\right)\stackrel{\sign}{\rightarrow} H^0(\sper k(x), 2^{j-1}\Z)$$ with the isomorphism $$H^0(\sper k(x), 2^{j-1}\Z) \stackrel{2}{\rightarrow} H^0(\sper k(x), 2^j\Z)$$ and as such is an isomorphism. Therefore we have an isomorphism between the kernels of the horizontal maps. We may identify the kernel of $\oplus \partial_{\pi}$ over $U$ with $\I^j(U)$ using Lemma \ref{lem:KernelIjresidue}, and we identify $\supp_{*}2^n\Z(U)$ with the kernel of $\oplus\beta_{\pi}$ over $U$ using Lemma \ref{lemma:KernelIsBeta}. This finishes the proof of Theorem \ref{thm:B}.
\end{proof}

\section{Higher signatures}\label{sec:HigherSign}
\subsection{Derived Witt groups}\mbox{}\\
The Witt group $W(X)$ is a part of a cohomology theory $W^n(X)$ for schemes called the derived Witt groups. Next, we briefly recall the definition of the derived Witt groups that we use throughout.
\subsection{Definition of the derived Witt groups}\mbox{}\\
Let $(\mathcal{E},\sharp)$ be an exact category with duality $\sharp$. The homotopy category $\mathcal{K}^b(\mathcal{E})$ of bounded chain complexes in $\mathcal{E}$ is a category having as objects bounded chain complexes in $\mathcal{E}$ and as morphisms the chain maps up to chain homotopy. The bounded derived category $D^b(\mathcal{E})$ is obtained from the homotopy category by formally inverting quasi-isomorphisms. The duality $\sharp$ on $\mathcal{E}$ induces a duality on the homotopy category $\mathcal{K}^b(\mathcal{E})$ and on the derived category $D^b(\mathcal{E})$. Let $\varpi$ denote the isomorphism to the double dual $\varpi:1\stackrel{\backsimeq}{\rightarrow}\sharp \sharp$ in $D^b(\mathcal{E})$ that is induced from the canonical one in $\mathcal{E}$. Then, $(D^b(\mathcal{E}), \sharp, \varpi, 1)$ is a triangulated category with duality. For a reference for these facts see %\cite[Section 3.1.3]{SchlichtingSedano},
\cite[Section 2.6]{BalmerTriII}.\\
Let $X$ be a scheme with $2$ invertible in the global sections $\mathcal{O}_X(X)$. Let $\textup{Vect}(X)$ denote the exact category of vector bundles on $X$, that is, the category of $\mathcal{O}_X$-modules which are locally free and of finite rank. For any vector bundle $\mathcal{E}$ on $X$, the usual duality $\mathcal{E}^{\sharp}:= \textup{Hom}_{\textup{Vect}(X)}(\mathcal{E},\mathcal{O}_X)$ defines a duality on $\textup{Vect}(X)$, making $(\textup{Vect}(X),\sharp)$ an exact category with duality. The \emph{derived Witt groups of $X$} are the Witt groups $W^n(D^b(\textup{Vect}(X)))$ of the triangulated category with duality $(D^b(\textup{Vect}(X)),\sharp, \varpi, 1)$ \cite[Definition 2.3]{BalmerTriII}. They will be denoted by $W^n(X):=W^n(D^b(\textup{Vect}(X)))$. 

\subsection{Coniveau spectral Sequence}\mbox{}\label{subsect:ConiveauspectralSequence}\\
Let $X$ be a regular, noetherian, separated, $\Zh$-scheme. In order to make clear what is the filtration and construction of the spectral sequence of Theorem \ref{thm:A}, in this section we recall the construction of the coniveau spectral sequence for the derived Witt groups due to P. Balmer and C. Walter \cite{BalmerWalter}.
\subsection{Filtration by codimension of support}\label{subsect:ConstructionGerstenWittCMX}\mbox{}\\
Let $D^b(X):=D^b(\textup{Vect}(X)).$ We recall the filtration on $D^b(X)$ by codimension of support. For any $\mathcal{O}_X$-module $M$, we denote by $\textup{supp}M$ the set of points $x\in X$ for which the localization $M_x$ is non-zero. For any complex $M_{\bullet}\in D^b(X)$, $\textup{supp}H_i(M_{\bullet})$ is a closed subspace of $X$. The support of a complex $M_{\bullet}$ is defined to be
\begin{equation*}
\textup{supp}M_{\bullet}:= \bigcup_{i\in \Z}\textup{supp}H_i(M_{\bullet})
\end{equation*}
and it is also a closed subspace of $X$ as it is a finite union, since complexes in $D^b(X)$ are bounded, of closed subspaces. Let $D^{b}_Z(X)\subset D^b(X)$ consist of those complexes $M_{\bullet}$ having $\textup{supp}M_{\bullet}\subset Z$. Recall that the codimension of a closed subspace $Z$ of $X$ is defined as
\begin{equation*}\label{Eqn:CodimensioDefinition}
\textup{codim}(Z):=\textup{min}_{\eta \in Z}\textup{dim}\mathcal{O}_{X,\eta}
\end{equation*}
where the minimum runs over the finitely many generic points $\eta\in Z$ of $Z$ (as $X$ is noetherian, so is $Z$, hence $Z$ has finitely many irreducible components).

For any integer $p\geq 0$, let $D^{p}(X)\subset D^b(X)$, or simply $D^p$, consist of those complexes having codimension of support greater than or equal to $p$, that is
\begin{equation*}\label{Eqn:CodimSupportFiltration}
D^{p}(X)=\{M_{\bullet}\in D^b(X) | \textup{codim}_X(\textup{supp}M_{\bullet})\geq p\}
\end{equation*}

For $p\geq0$, we have short exact sequences \cite[Theorem 3.1 and proof]{BalmerWalter} of triangulated categories with duality
\begin{equation*}
 D^{p+1} \rightarrow D^{p} \rightarrow D^p/D^{p+1}
\end{equation*}
which induce \cite[Theorem 3.1 and proof]{BalmerWalter}
the long exact sequence of groups
\begin{equation*}
\cdots \rightarrow W^{p+q}(D^{p+1}) \stackrel{i_{p+1,q-1}}{\rightarrow} W^{p+q}(D^{p}) \stackrel{j_{p,q}}{\rightarrow} W^{p+q}(D^{p}/D^{p+1}) \stackrel{k_{p,q}}{\rightarrow} W^{p+q+1}(D^{p+1})\stackrel{i_{p+1, q}}{\rightarrow}\cdots
\end{equation*}
where we index the maps based on the indices of their respective domain. The map $i$ is induced by the inclusion $D^{p+1}\rightarrow D^p$, $j$ by the quotient $D^p\rightarrow D^{p}/D^{p+1}$, and $k$ is the connecting morphism.

From the long exact sequences we obtain an exact couple by setting $\text{E}_1^{p,q}:=W^{p+q}(D^p/D^{p+1})$, $\text{D}_1^{p,q}:=W^{p+q}(D^{p})$, and taking the differential to be $d^{p,q}:=j_{p+1,q}\circ k_{p, q}$. By the well-known method of Massey's exact couples, this exact couple determines the spectral sequence below
  \begin{equation*}\label{eqn:ConiveauCoherentWitt}
\text{E}_1^{p,q}:= W^{p+q}(D^p/D^{p+1}) \Rightarrow W^{p+q}(X)
  \end{equation*}
with abutment the derived Witt groups. The differential $d_r$ on the $r$-th page of this spectral sequence has bidegree $(r,1-r)$. By saying that the abutment is the derived Witt groups, we mean that there is a filtration on the $n$th derived Witt group $W^{n}(X)$
\begin{equation*}
 \cdots \subset \text{F}^{3,n-3} \subset \text{F}^{2,n-2}\subset \text{F}^{1,n-1}\subset W^n(X)
\end{equation*} 
where  
\begin{equation*}
\text{F}^{p,q}:=\textup{im}(W^{p+q}(D^p)\rightarrow W^{p+q}(X)).
\end{equation*}
When the dimension of $X$ is finite the spectral sequence \emph{strongly converges}, by which we mean that the filtration on the abutment is a finite filtration, and that the terms $\text{F}^{p,q}$ form exact sequences of groups
\begin{equation*}
0\rightarrow \text{F}^{p+1,q-1}\rightarrow \text{F}^{p,q} \rightarrow \text{E}^{p,q}_{\infty}\rightarrow 0
\end{equation*}
where $\text{E}^{p,q}_{\infty}$ is the stable term, that is to say, for some $r$ sufficiently large, $\text{E}^{p,q}_r=\text{E}^{p,q}_{r+n}$ for $n\geq 0$ (because the dimension of $X$ is finite and thus the differentials $d_r$ are eventually zero for $r$ sufficiently large) and this group is labeled $E^{p,q}_{\infty}$.

\subsection{The coniveau spectral sequence after inverting 2}\label{subsect:ConiveauAfterInverting2}\mbox{}\\
Since localization of $\Z$-modules is exact, we obtain from the exact couple of the coniveau spectral sequence an exact couple  $\text{HE}_1^{p,q}:=W^{p+q}(D^p/D^{p+1})[\nicefrac{1}{2}]$, $\text{HD}_1^{p,q}:=W^{p+q}(D^{p})[\nicefrac{1}{2}]$. This determines the spectral sequence
  \begin{equation*}
\text{HE}_1^{p,q}:= W^{p+q}(D^p/D^{p+1})[\nicefrac{1}{2}] \Rightarrow W^{p+q}(X)[\nicefrac{1}{2}]
  \end{equation*}
with abutment the derived Witt groups with 2 inverted. Again the differential $d_r$ on the $r$-th page of this spectral sequence has bidegree $(r,1-r)$. The filtration on the abutment $W^{n}(X)[\nicefrac{1}{2}]$ is
\begin{equation*}\label{Eqn:FiltrationOnAbutment}
\cdots \subset \text{HF}^{3,n-3} \subset \text{HF}^{2,n-2}\subset \text{HF}^{1,n-1}\subset W^n(X)[\nicefrac{1}{2}]
\end{equation*} 
where  
\begin{equation*}
\text{HF}^{p,q}:=\textup{im}(W^{p+q}(D^p)[\nicefrac{1}{2}]\rightarrow W^{p+q}(X)[\nicefrac{1}{2}]).
\end{equation*}
The following Lemma recalls facts about the coniveau spectral sequence due to P. Balmer and C. Walter. We only give a proof to address the case with 2 inverted.
\begin{lemma}\label{lem:GerstenComplexZariskiCoh}
Let $F$ be a field of characteristic different from 2. Let $X$ be a regular, noetherian, separated $F$-scheme. For $q$ not congruent to 0 modulo 4, the $\textup{E}_2^{*,q}$-terms of the coniveau spectral sequence and the coniveau spectral sequence after inverting 2 are zero. For $q$ congruent to 0 modulo 4, in the coniveau spectral sequence we have that $\textup{E}_2^{p,q}=H^p_{Zar}(X,\W)$ and in the coniveau spectral sequence after inverting 2, $\textup{HE}_2^{p,q}=H^p_{Zar}(X,\W[\nicefrac{1}{2}])$.
\end{lemma}
\begin{proof}
Recall that the Gersten conjecture is known for any regular local ring containing a field $F$ of characteristic different from 2 \cite[Theorem 6.1]{GillePaninBalmerWalter}. Hence, the description of the $\text{E}_2^{p,q}$-terms of the coniveau spectral sequence is \cite[Theorem 7.2]{BalmerWalter} together with \cite[Lemma 4.2]{GillePaninBalmerWalter}.

In order verify our claim about inverting 2, for each open $U$ in $X$, let $\mathcal{G}\mathcal{W}\mathcal{C}^p$ denote the sheaf $U\mapsto \text{E}_1^{p,0}(U)$; the sheaves $\mathcal{G}\mathcal{W}\mathcal{C}^p$ form a complex of sheaves $\mathcal{G}\mathcal{W}\mathcal{C}$ which is a flasque resolution of the Witt sheaf $\W$ \cite[Proof of Lemma 4.2]{GillePaninBalmerWalter}. The sheaf $U\mapsto \text{HE}_1^{p,0}(U)$ is exactly $\mathcal{G}\mathcal{W}\mathcal{C}^p[\nicefrac{1}{2}]$ and $\mathcal{G}\mathcal{W}\mathcal{C}^p[\nicefrac{1}{2}]$ is a direct limit of sheaves (Lemma \ref{lem:Inverting2CommutesWithCohomology} $(i)$). Since direct limits are exact, the sheaves $\mathcal{G}\mathcal{W}\mathcal{C}^p[\nicefrac{1}{2}]$ form a complex which is a flasque resolution of the Witt sheaf $\W[\nicefrac{1}{2}]$. Hence $\text{HE}_2^{p,q}=H^p_{Zar}(X,\W[\nicefrac{1}{2}])$ as claimed.
\end{proof}

\subsection{Definition of the higher signatures}
\label{subsec:DefinitionHigherSignature}
Let $F$ be a field of characteristic different from 2. Let $X$ be a regular, noetherian, separated $F$-scheme of finite Krull dimension. A main idea of the article is to use the group homomorphism 
\begin{equation*}\label{eqn:IntroCohomologySignature}
H^i_{Zar}(X,\W)\stackrel{\shfsign}{\rightarrow} H^i(X_r,\Z)
\end{equation*}
induced by $\shfsign$ to define ``higher'' signature maps. Recall that the derived Witt groups $W^i(X)$ are 4-periodic, $W^i(X)\simeq W^{4+i}(X)$, so we only define higher signatures for $W^0(X), W^1(X), W^2(X)$ and $W^3(X)$.

If $i$ is any integer $0\leq i \leq 3$, it follows from the description of the $\text{E}_2$-page of the coniveau spectral sequence in Lemma \ref{lem:GerstenComplexZariskiCoh} that the stable terms $\text{E}_{\infty}^{i,0}$ are subgroups of $H^i_{Zar}(X,\W)$ since they equal the kernel of a differential leaving $H^i_{Zar}(X, \W)$. Let $\text{E}_{\infty}^{i,0} \hookrightarrow H^i_{Zar}(X,\W)$ denote the inclusion map. The composition $W^i(X)\twoheadrightarrow \text{E}_{\infty}^{i,0} \hookrightarrow H^i_{Zar}(X,\W)$ is the edge map in the coniveau spectral sequence.

\begin{definition}\label{def:HigherSign}
Let $F$ be a field of characteristic different from 2. Let $X$ be a noetherian, regular, separated $F$-scheme of finite Krull dimension and let $0\leq i \leq 3$. The composition 
\begin{equation*}
W^i(X)\twoheadrightarrow \textup{E}_{\infty}^{i,0} \hookrightarrow H^i_{Zar}(X,\W)\stackrel{\mathcal{S}ign}{\rightarrow} H^i(X_r,\Z)
\end{equation*}
determines a group homomorphism
\begin{equation*}
\isign:W^i(X)\rightarrow H^i(X_r,\Z)
\end{equation*}
that we call the \emph{$i$-th global signature}. We will also refer to these maps as the \emph{higher signatures of} $X$. 
\end{definition}
When $i=0$, the map $\textup{sign}^0$ agrees with the global signature $\sign:W(X)\rightarrow H^0(X_r,\Z)$ in the introduction because $\sign$ factors through the global sections $\W(X)$ of the Witt sheaf.

\subsection{Atiyah-Hirzebruch spectral sequence for derived Witt groups with 2 inverted}\label{subsect:AtiyahHirz}
Recall that for any abelian group $M$, the real cohomology $H^{*}(X_r, M)$ of a separated $\R$-variety $X$ coincides with the singular cohomology $H^{*}(X(\R),M)$ of the real points $X(\R)$ (Section \ref{subsubsect:RealCoh}). Hence, one may think of the following theorem as giving an Atiyah-Hirzebruch type spectral sequence for derived Witt groups with 2 inverted.
\begin{theorem}\label{thm:A}
Let $X$ be a noetherian, regular, separated $F$-scheme, $F$ a field of characteristic different from 2. The coniveau spectral sequence for the derived Witt groups, after inverting 2, determines a spectral sequence 
\begin{equation*}
\textup{E}_2^{p,q}=\begin{cases} H^p(X_r,\Zh)& \textup{if } q\equiv 0 \textup{ mod } 4\\
						0&	\textup{otherwise}
			\end{cases} 
			\Longrightarrow W^{p+q}(X)[\nicefrac{1}{2}]
\end{equation*}
abutting to the derived Witt groups with 2 inverted. For $r\geq1$, the differentials $d_r$ are of bidegree $(r,r-1)$. The groups $H^p(X_r,\Zh)$ are the cohomology of the real spectrum with coefficients in $\Zh$. For integers $0\leq i \leq 3$, the higher global signatures $\isign$, after inverting 2, are edge maps in this spectral sequence. When the dimension of $X$ is finite, this spectral sequence strongly converges.
\end{theorem}
\begin{proof}
Let $X$ be a regular, noetherian, separated $F$-scheme. As described in Section \ref{subsect:ConiveauAfterInverting2} and Lemma \ref{lem:GerstenComplexZariskiCoh}, after inverting 2 the coniveau spectral sequence determines a spectral sequence
\begin{equation*}
\text{HE}_2^{p,q}=\begin{cases} H^p_{Zar}(X,\W[\nicefrac{1}{2}])& \textup{if } q\equiv 0 \textup{ mod } 4\\
						0&	\textup{otherwise}
			\end{cases} 
			\Longrightarrow W^{p+q}(X)[\nicefrac{1}{2}]
\end{equation*}
abutting to the derived Witt groups with 2 inverted.
Using Theorem \ref{thm:0}, we have that $\mathcal{S}ign$, after inverting 2, determines an isomorphism from the group $H^p_{Zar}(X,\W[\nicefrac{1}{2}])$ to the group $H^p(X_r,\Z[\nicefrac{1}{2}])$. This finishes the proof.  
\end{proof}

The next corollary immediately follows upon considering the structure of the $\textup{E}_2$-page
\begin{corollary}\label{cor:HigherSignIsomorphism}
If $\dim X \leq 3$, then the spectral sequence of Theorem \ref{thm:A} collapses on the $\text{E}_2$-page and the higher signature maps induce isomorphisms of groups
$$W^0(X)[\nicefrac{1}{2}] \stackrel{\textup{sign}^0}{\simeq} H^0(X_r,\Zh)$$
$$W^1(X)[\nicefrac{1}{2}] \stackrel{\textup{sign}^1}{\simeq} H^1(X_r,\Zh)$$
$$W^2(X)[\nicefrac{1}{2}] \stackrel{\textup{sign}^2}{\simeq} H^2(X_r,\Zh)$$
$$W^3(X)[\nicefrac{1}{2}] \stackrel{\textup{sign}^3}{\simeq} H^3(X_r,\Zh)$$
\end{corollary}  

If $\dim X=4$, then the map $$W^0(X)[\nicefrac{1}{2}] \stackrel{\textup{sign}^0}{\simeq} H^0(X_r,\Zh)$$ of Corollary \ref{cor:HigherSignIsomorphism} is not in general an isomorphism. For instance if $$X=\spec \frac{\R[T_1, T_2, T_3, T_4, T_5]}{T_1^2+T_2^2+T_3^2+T_4^2+T_5^2-1}$$ then it follows from Corollary \ref{cor:OddTorsionDerived} that the kernel of $$W^0(X)[\nicefrac{1}{2}] \stackrel{\textup{sign}^0}{\rightarrow} H^0(X_r,\Zh)$$ is non-trivial, isomorphic to $H^4(X(\R),\Zh)\simeq \Zh$. In particular, the kernel of the global signature is not torsion in general. This had already been addressed in \cite[start of Section 2]{Mahe}, using rings which are not of finite type over a field. There L. Mah\'e remarked that one might ask if the kernel of the global signature is torsion in general and stated that Knebusch has answered this question negatively by providing the example of the ring $C(S^4)$ of real-valued continuous functions on the real sphere $S^4$ \cite[p.189]{KnebuschQueens}. 

As $H^0(X_r,\Zh)$ is torsion free, the next corollary follows from the previous one. It generalizes the classic fact that the Witt group of a field may only have 2-primary torsion.
\begin{corollary}
If $\dim X \leq 3$, then torsion in $W(X)$ is 2-primary.
\end{corollary}
 
The following result gives sufficient conditions on when non-torsion elements or odd-torsion elements can be present in the kernel of the higher global signatures. 
Using Proposition \ref{prop:Differentials} which appears later in this article, one can make the sequence in the next corollary exact on the right when $X$ is affine.
\begin{corollary}\label{cor:OddTorsionDerived}
If $X$ is a separated scheme that is smooth and of finite type over $\R$ and $4\leq \dim X \leq 7$, then there is an exact sequence of groups
 \begin{equation}
0\rightarrow H^{4}_{sing}(X(\R), \Zh)\rightarrow W(X)[\nicefrac{1}{2}]\stackrel{\sign}{\rightarrow}H^0_{sing}(X(\R), \Zh)
 \end{equation}
\end{corollary}
\begin{proof}
Consider the spectral sequence of the Theorem \ref{thm:A}. Using the description of the first page of this spectral sequence from Lemma \ref{lem:GerstenComplexZariskiCoh} together with the hypothesis on the dimension of $X$, we find that the differentials leaving $H^{4}_{sing}(X(\R), \Zh)$ are all trivial. So $H^{4}_{sing}(X(\R), \Zh)=\text{HE}_{\infty}^{4,-4}$ and $\text{HE}_{\infty}^{4,-4}=\text{HF}^{4, -4}$. Hence the filtration of the spectral sequence gives a short exact sequence of groups
\begin{equation*}
0\rightarrow H^{4}_{sing}(X(\R), \Zh)\rightarrow W(X)[\nicefrac{1}{2}]\rightarrow \text{HE}_{\infty}^{0,0} \rightarrow 0
\end{equation*}
and since $\text{HE}_{\infty}^{0,0}$ is the kernel of a differential leaving $H^{0}(X(\R),\Zh)$, it is a subgroup of $H^{0}(X(\R),\Zh)$. By composing the exact sequence above with the inclusion $\text{HE}_{\infty}^{0,0}\hookrightarrow H^{0}(X(\R),\Zh)$ we obtain the exact sequence in the statement of the theorem. 
\end{proof}

\begin{corollary}\label{cor:4Manifold2Primary}
Suppose that $X$ is a separated scheme that is smooth and of finite type over $\R$ and that $X(\R)$, when equipped with the Euclidean topology, is compact as a subspace of Euclidean space. If $\dim X \leq 4$, then torsion in $W(X)$ is 2-primary.
\end{corollary}
\begin{proof}
As $X(\R)$ is compact, we may use Poincar\'e duality to obtain that the group $H^4_{sing}(X(\R),\Z)$ has no odd-torsion. Using the exact sequence of Corollary \ref{cor:OddTorsionDerived}, it follows that the Witt group $W(X)$ cannot contain odd-torsion.
\end{proof}

In Corollary \ref{cor:4Manifold2Primary}, the bound on the dimension is sharp. For any prime $p$, one can find a smooth algebraic variety $X$ with $p$-torsion in its Witt group $W(X)$, $X(\R)$ compact, and $\dim X =5$. To do so, for every prime $p$ and integer $n\geq 1$, let $L^{2n+1}(p)$ denote the topological space formed by taking the quotient space of the sphere $S^{2n+1}$, considered as a subspace of $\mathbb{C}^{n+1}$, by the action of $p$th roots of unity $\mu_{p}$ given by $(z_0, z_1, \cdots, z_n)\mapsto (\mu_{p}z_0, \mu_{p}z_1, \cdots, \mu_{p}z_n)$. The resulting topological space is called the \emph{compact Lens space}. It is a compact $C^{\infty}$-manifold of dimension $2n+1$ as a real manifold. The Nash-Tognoli Theorem states that every compact $C^{\infty}$ submanifold $M\subset\R^m$ is diffeomorphic to the real points $X(\R)$ of a smooth algebraic variety $X$ over the real numbers $\R$ \cite[Theorem 14.1.10]{Coste}. Let $X_{5,p}$ be such a variety with $X_{5,p}(\R)$ diffeomorphic to $L^5(p)$. Note the Krull dimension of $X_{5,p}$ must be 5 \cite[Proposition 2.8.14]{Coste}. The degree 4 singular cohomology group of the lens space is $H^4_{sing}(L^5(p), \Z)\simeq \Z/p$. Applying Corollary \ref{cor:OddTorsionDerived} we obtain that $W(X)[\nicefrac{1}{2}]$ contains an element $\alpha$ which is $p$-torsion. It follows that, for some integer $k$, $2^k\alpha\in W(X)$ is $p$-torsion.

For real varieties, one can bound the ranks of the Witt groups in terms of Betti numbers using the spectral sequence. Consequently, using the Milnor-Thom bound \cite{Milnor}, one may bound the direct sum $\oplus_{i=0}^{i=3} W^i(X)$ in terms of the degree of the polynomials defining $X$.
 
\begin{corollary}\label{cor:Betti}
Let $X$ be a separated scheme that is smooth and of finite type over $\R$. Let $b_i$ denote the $i$th Betti number of $X(\R)$. For each integer $0\leq j \leq 3$, let $N_j=\floor*{\frac{d-j}{4}}$, where we set $N_j=0$ if $d-j\leq 0$. Then the ranks of the Witt groups of $X$ are bounded in terms of the Betti numbers:
$$\hspace{-3.5mm}\textup{rank}W^0(X) \leq \sum^{N_0}_{i=0} b_{4i}$$
$$\textup{rank}W^1(X) \leq \sum^{N_1}_{i=0} b_{4i+1}$$
$$\textup{rank}W^2(X) \leq \sum^{N_2}_{i=0} b_{4i+2}$$
$$\textup{rank}W^3(X) \leq \sum^{N_3}_{i=0} b_{4i+3}$$
\end{corollary}
\begin{proof}
Let $X$ be a separated scheme that is smooth and of finite type over $\R$. Let $d$ denote the Krull dimension of $X$. We explain how to find the bound on the rank of $W^0(X)$. Finding the bound on $W^1(X), W^2(X)$, and $W^3(X)$ is identical. From the construction of the spectral sequence of Theorem \ref{thm:A}, the Witt group $W^0(X)[\nicefrac{1}{2}]$ lives in the middle of a short exact sequence
\begin{equation}\label{eqn:FirstFiltrationSequence}
0\rightarrow \text{HF}^{1,-1}\rightarrow W^0(X)[\nicefrac{1}{2}] \rightarrow \text{HE}^{0,0}_{\infty}\rightarrow 0
\end{equation}
where $\text{HF}^{1,-1}$ is the first group in a finite filtration on $W^0(X)[\nicefrac{1}{2}]$
\begin{equation*}
\emptyset=\text{HF}^{d+1, -d-1} \subset \text{HF}^{d,-d} \subset \cdots \subset \text{HF}^{3,-3} \subset \text{HF}^{2,-2}\subset \text{HF}^{1,-1}\subset W^0(X)[\nicefrac{1}{2}]
\end{equation*} 
The other groups belong to short exact sequences, for every pair of integers $p,q$ with $p+q=0$,
\begin{equation*}
0\rightarrow \text{HF}^{p+1,q-1}\rightarrow \text{HF}^{p,q} \rightarrow \text{HE}^{p,q}_{\infty}\rightarrow 0
\end{equation*}
After tensoring these short exact sequences with $\Q$, they all split, and so we have that  
$$\textup{rank}W^0(X) = \sum^{d}_{k=0} \textup{rank}\text{HE}_{\infty}^{k,-k} $$

Note that $\text{HE}^{k,-k}_{\infty} =0$ when $k$ is not congruent to 0 mod 4 and when $k>d$ (Lemma \ref{lem:GerstenComplexZariskiCoh}). So we have, letting $N_0=\floor*{\frac{d}{4}}$,
$$\textup{rank}W^0(X) = \sum^{N_0}_{k=0} \textup{rank}\text{HE}_{\infty}^{4k,-4k} \leq \sum^{N_0}_{k=0} b_{4k} $$
where the inequality is obtained using the fact that the groups $\text{HE}_{\infty}^{4k,-4k}$ are subquotients of $H^{4k}_{sing}(X(\R),\Zh)$.
\end{proof}

Finally, we demonstrate how Mah\'e s result on the cokernel of the total signature being 2-primary is equivalent to the differentials leaving $H^0(X_r, \Zh)$ in this spectral sequence being zero. 
\begin{proposition}\label{prop:Differentials}
If $X$ is affine and satisfies the hypotheses of Theorem \ref{thm:A}, then the differentials leaving $H^0(X,\Zh)$ in the spectral sequence of Theorem \ref{thm:A} are zero. Equivalently, the images of the differentials leaving $H^0(X,\W)$ in the coniveau spectral sequence for Witt groups are 2-primary torsion.
\end{proposition}
\begin{proof}
Let $X=\spec A$ be an affine, noetherian, regular, separated $F$-scheme. From the definition of the global signature it follows that the diagram below commutes
\begin{equation*}
\xymatrix{
W(X)[\nicefrac{1}{2}] \ar[d] \ar^{\sign}[r] & H^0(X_r,\Z[\nicefrac{1}{2}])\\
\text{HE}_{\infty}^{0,0} \ar[r] & H^0_{Zar}(X,\W[\nicefrac{1}{2}]) \ar[u]
}
\end{equation*}
where the rightmost vertical map is an isomorphism by Theorem \ref{thm:0}, and the bottom horizontal map is injective (See Section \ref{subsec:DefinitionHigherSignature}). Using Mah\'e's Theorem \cite[Theoreme 3.2]{Mahe}, which states that the cokernel of the global signature of an affine scheme is 2-primary torsion, we have that $\sign:W(X)[\nicefrac{1}{2}]\rightarrow H^0_{Zar}(X_r,\Z)[\nicefrac{1}{2}]$ is a surjection. Hence the injection $\text{HE}_{\infty}^{0,0} \rightarrow H^0(X,\W[\nicefrac{1}{2}])$ is an isomorphism. Since $\text{HE}_{\infty}^{0,0}$ is a subgroup of the kernel of every differential leaving $H^0_{Zar}(X,\W[\nicefrac{1}{2}])$, it follows that all such differentials are zero. Therefore the images of the differentials leaving $H^0_{Zar}(X,\W)$ in the coniveau spectral sequence must be 2-primary torsion groups. This finishes the proof of the corollary.
\end{proof}

\section{Bounding the order of torsion in the Witt group}\label{sec:BoundingTorsion}
The aim of this section is to prove the following theorem. 
\begin{theorem}\label{thm:AppB}
Let $F$ be a field of characteristic different from 2. Let $X$ be an excellent, integral, noetherian, regular, separated $F$-scheme of finite Krull dimension and of finite virtual cohomological 2-dimension. If the real cohomology groups $H^p(X_r, \Z)$ are finitely generated, then:
\begin{enumerate}[(i)]
\item the torsion in the Witt groups $W^i(X)$ is of bounded order;
\item the rank of the Witt group $W^i(X)$ is finite;
\item each group $W^i(X)$ splits as a direct sum $W^i(X)\simeq W^i_{tor}(X)\oplus \Z^r$ where $W^i_{tors}(X)$ denotes the subgroup of torsion elements in $W^i(X)$ and $r$ is the rank of $W^i(X)$;
\item $W^i_{tors}(X)$ is a direct sum of cyclic groups.
\end{enumerate} 
\end{theorem}

To prove that the derived Witt groups have bounded torsion, it is not sufficient to prove bounded torsion for the $\text{E}_2$-terms of the coniveau spectral sequence. This is because a complex of abelian groups with bounded torsion may have cohomology groups that possess torsion elements of any order. 

Instead we prove that the $\text{E}_2$-terms belong to a smaller Serre subcategory which we introduce next.
 
\subsection{Serre subcategory}
\label{subsect:SerreSubcat}
Recall that a subcategory $\mathcal{C}$ of the category of abelian groups is a Serre subcategory if for any short exact sequence of abelian groups
\begin{equation}\label{eqn:SES}
0\rightarrow A\rightarrow B\rightarrow C \rightarrow 0
\end{equation}
we have that $B$ is in $\mathcal{C}$ if and only if $A$ and $C$ are in $\mathcal{C}$.

\begin{lemma}\label{lem:Serre}
Let $\mathcal{C}$ denote the subcategory of the category of abelian groups consisting of all abelian groups $G$ having the following two properties:
\begin{enumerate}[(i)]
\item there exists a nonzero integer $n$ such that $nG_{tors}$ is trivial, where $G_{tors}$ denotes the subgroup of torsion elements in $G$;
\item the quotient $G_{red}:=G/G_{tors}$ is free and of finite rank, that is, $G_{red}$ is either trivial or is isomorphic to $\Z^m$ for some positive integer $m$.
\end{enumerate}
Then $\mathcal{C}$ is a Serre subcategory and any group $G$ in $\mathcal{C}$ has the property that the torsion part $G_{tors}$ is a direct sum (possibly infinite) of cyclic groups and $G$ splits $G\simeq G_{tors}\oplus G_{red}$. 
\end{lemma}
\begin{proof}
The two final statements about the properties of a group $G$ in $\C$ are facts from group theory valid for any abelian group having bounded torsion. To prove that $\C$ is a Serre subcategory, let $A, B$ and $C$ be abelian groups forming a short exact sequence as in \eqref{eqn:SES}.
In the one direction, suppose that $A$ and $C$ are in $\C$ with $mA_{tors}=0$, $nC_{tors}=0$, $A_{red}\simeq \Z^j$, and $C_{red}\simeq \Z^k$. As elements of $A_{tors}$ are of order at most $m$, and elements of $C_{tors}$ are of order at most $n$, then it follows that elements of $B_{tors}$ are of order at most $mn$, hence $mnB_{tors}=0$. To prove that $B_{red}$ is free and of finite rank, consider the commutative diagram below
\begin{equation}\label{eqn:BtoCDiagram}
\xymatrix{
0\ar[r] & B_{tors} \ar[r] \ar[d]                   & B \ar[r]	\ar[d]				& B_{red} \ar[d] \ar[r] & 0 \\    
0\ar[r] & C_{tors} \ar[r] 						  & C \ar[r]						& \Z^k\simeq C_{red}\ar[r] & 0\\	
}
\end{equation}
and apply the snake lemma to obtain the exact sequence below,
\begin{equation*}
0\rightarrow A_{tors} \rightarrow A \rightarrow \textup{ker}(B_{red}\rightarrow \Z^k)\rightarrow \textup{cok}(B_{tors}\rightarrow C_{tors})\rightarrow 0
\end{equation*}
Recall $A/A_{tors}\simeq \Z^j$, so from the previous exact sequence we obtain the short exact sequence below.
\begin{equation}\label{ses:BRed}
0 \rightarrow \Z^j \rightarrow \textup{ker}(B_{red}\rightarrow \Z^k) \rightarrow \textup{cok}(B_{tors}\rightarrow C_{tors}) \rightarrow 0
\end{equation} 
Multiplying by $n$ we obtain the commutative diagram below.
\begin{equation*}
\xymatrix{
0 \ar[r] &\Z^j \ar[r] \ar^{n}[d] & \textup{ker}(B_{red}\rightarrow \Z^k) \ar[r] \ar^{n}[d] & \textup{cok}(B_{tors}\rightarrow C_{tors}) \ar^{n}[d] \ar[r] & 0\\    
0 \ar[r] &\Z^j \ar[r] 						  & \textup{ker}(B_{red}\rightarrow \Z^k) \ar[r]						& \textup{cok}(B_{tors}\rightarrow C_{tors}) \ar[r] & 0\\	
}
\end{equation*}
From $n\C_{tors}=0$ we have $n\textup{cok}(B_{tors}\rightarrow C_{tors})=0$, so when we apply the snake lemma we obtain the injection 
\begin{equation*}
0\rightarrow \textup{cok}(B_{tors}\rightarrow C_{tors}) \rightarrow (\Z/n)^j 
\end{equation*}
so $\textup{cok}(B_{tors}\rightarrow C_{tors})$ is finite. Then, from short exact sequence \eqref{ses:BRed}, we conclude that the torsion free group $B_{red}$ is finitely generated, hence is free and of finite rank.

In the other direction, suppose that $B$ is in $\mathcal{C}$. Then, $A_{tors}$ has bounded order since $A_{tors}$ injects into $B_{tors}$, and $A_{red}$ is free and of finite rank since $A_{red}$ injects into $B_{red}$ and every subgroup of a free abelian group is free. Furthermore, as $B$ surjects onto $C$, it follows that $B_{red}\simeq \Z^m$ surjects onto $C_{red}$, hence $C_{red}$ finitely generated and torsion free. As such, $C_{red}$ is free and of finite rank.\\ To prove that $C_{tors}$ is of bounded order, apply the snake lemma to Diagram \eqref{eqn:BtoCDiagram} to obtain that the cokernel of $B_{tors}\rightarrow C_{tors}$ is finite because the finitely generated group $ker(B_{red}\rightarrow \Z^k)$ surjects onto it, hence $C_{tors}$ lives in the middle of a short exact sequence with two groups that are both torsion of bounded order. As discussed earlier, it follows that $C_{tors}$ is torsion of bounded order. This completes the proof.
\end{proof}  

\begin{theorem}\label{thm:E}
Let $F$ be a field of characteristic different from 2. Let $X$ be an excellent, integral, noetherian, regular, separated $F$-scheme. Assume that the real cohomology groups $H^p(X_r,\Z)$ are finitely generated for all $p\geq 0$ and that $\vcd(X)$ is finite, say $\vcd(X)\hspace{-1mm}=\hspace{-1mm}s$. Then the cohomology groups $H^p_{Zar}(X,\W)$ are in the Serre subcategory of Lemma \ref{lem:Serre}.
\end{theorem}
\begin{proof}
Let $\vcd(X)\hspace{-1mm}=\hspace{-1mm}s$ and let $\C$ denote the Serre subcategory of Lemma \ref{lem:Serre}. For every $j\geq 1$ there is a long exact sequence in cohomology
\begin{equation*}
\cdots \rightarrow H^i_{Zar}(X,\I^j)\rightarrow H^i_{Zar}(X,\W)\rightarrow H^i_{Zar}(X,\W/\I^j))\rightarrow \cdots 
\end{equation*}
For $j\geq 1$, the map $\W/\I^j\stackrel{2^{j}}{\rightarrow}\W/\I^j$ is zero, so the map it induces in cohomology is zero. Therefore the groups $H^i_{Zar}(X,\W/\I^j))$ are torsion, with $2^jH^i_{Zar}(X,\W/\I^j))=0$, and hence belong to $\C$. For $j\geq s+1$, using Theorem \ref{thm:B} (iii) and the hypothesis on finite generation of real cohomology, we have that the groups $H^i_{Zar}(X,\I^j)$ are finitely generated. Therefore, when $j\geq s+1$, both $H^i_{Zar}(X,\I^j)$ and $H^i_{Zar}(X,\W/\I^j))$ belong to $\C$, hence so does $H^i_{Zar}(X,\W)$.
\end{proof}

\subsection{Proof of Theorem \ref{thm:AppB}}
Let $F$ be a field of characteristic different from 2. Let $X$ be an excellent, integral, noetherian, regular, separated $F$-scheme having finite Krull dimension. Assume that the real cohomology groups $H^p(X_r,\Z)$ are finitely generated for all $p\geq 0$ and that $\vcd(X)$ is finite. Recall that if a spectral sequence converges strongly and the groups on the $\text{E}_2$-page lie in a Serre subcategory, then the $\text{E}_{\infty}$-terms, and hence the abutment, lie in the same subcategory. From Lemma \ref{lem:GerstenComplexZariskiCoh}, we have that the entries on the $\text{E}_2$-page of the coniveau spectral sequence are either 0 or equal to $H^p_{Zar}(X,\W)$ for some $p\geq 0$. By Theorem \ref{thm:E}, the groups $H^p_{Zar}(X,\W)$ lie in the Serre subcategory of Lemma \ref{lem:Serre}. Therefore, the abutment $W^i(X)$ does as well. This finishes the proof of Theorem \ref{thm:AppB}.

\bibliographystyle{spmpsci}      % mathematics and physical sciences
\bibliography{GlobalSignature}   % name your BibTeX data base

\end{document}